\documentclass[11pt]{article}
\usepackage{amsmath,amsfonts,amssymb,amsthm,fancyhdr,bm}
\usepackage[hidelinks]{hyperref}
\usepackage{fullpage}
\usepackage[dvipsnames]{xcolor}
\usepackage[color=Purple!50!white]{todonotes}

\allowdisplaybreaks

\newtheorem{thm}{Theorem}[section]
\newtheorem*{thm*}{Theorem}
\newtheorem*{thm:tripleexp}{Theorem \ref{thm:tripleexp}}
\newtheorem*{thm:quasirandomness}{Theorem \ref{thm:quasirandomness}}
\newtheorem*{thm:strongquasirandom}{Theorem \ref{thm:strongquasirandom}}
\newtheorem{lem}[thm]{Lemma}

\newtheorem{cor}[thm]{Corollary}
\theoremstyle{definition}
\newtheorem{Def}[thm]{Definition}
\newtheorem*{rem}{Remark}
\newtheorem{numrem}[thm]{Remark}

\DeclareMathOperator\pr{Pr}
\DeclareMathOperator\ext{ext}
\DeclareMathOperator\codeg{codeg}

\newcommand\ol[1]{\overline{#1}}

\newcommand\up[1]{^{(#1)}}

\newcommand\wt[1]{\widetilde{#1}}

\newcommand\Q{\mathcal{Q}}

\newcommand\N{\mathbb{N}}
\newcommand\R{\mathbb{R}}
\newcommand\E{\mathbb{E}}
\renewcommand\P{\mathcal{P}}

\newcommand{\pardiff}[2]{\mathchoice{\frac{\partial #1}{\partial #2}}{\partial #1/\partial #2}{\partial #1/\partial #2}{\partial #1/\partial #2}}

\title{Ramsey Numbers of Books and Quasirandomness}
\author{David Conlon\thanks{Department of Mathematics, California Institute of Technology, Pasadena, CA 91125, USA. Email: {\tt dconlon@caltech.edu}. Research supported by ERC Starting Grant 676632 and NSF Award DMS-2054452.} \and Jacob Fox\thanks{Department of Mathematics, Stanford University, Stanford, CA 94305, USA. Email: {\tt jacobfox@stanford.edu}. Research supported by a Packard Fellowship and by NSF Career Award DMS-1352121.} \and Yuval Wigderson\thanks{Department of Mathematics, Stanford University, Stanford, CA 94305, USA. Email: {\tt yuvalwig@stanford.edu}. Research supported by NSF GRFP Grant DGE-1656518.}}
\date{}

\begin{document}
\maketitle
\begin{abstract}
	The \emph{book graph} $B_n \up k$ consists of $n$ copies of $K_{k+1}$ joined along a common $K_k$. The Ramsey numbers of $B_n \up k$ are known to have strong connections to the classical Ramsey numbers of cliques. Recently, the first author determined the asymptotic order of these Ramsey numbers for fixed $k$, thus answering an old question of Erd\H os, Faudree, Rousseau, and Schelp. In this paper, we first provide a simpler proof of this theorem. Next, answering a question of the first author, we present a different proof that avoids the use of Szemer\'edi's regularity lemma, thus providing much tighter control on the error term. Finally, we prove a conjecture of Nikiforov, Rousseau, and Schelp by showing that all extremal colorings for this Ramsey problem are quasirandom.
\end{abstract}

\section{Introduction}
Given two graphs $H_1$ and $H_2$, their \emph{Ramsey number} $r(H_1,H_2)$ is the minimum $N$ such that any red/blue coloring of the edges of the complete graph $K_N$ contains a red copy of $H_1$ or a blue copy of $H_2$. Ramsey's theorem asserts that $r(H_1,H_2)$ is finite for all graphs $H_1$ and $H_2$. In the special case where $H_1$ and $H_2$ are the same graph $H$, we write $r(H)$ rather than $r(H,H)$. Though Ramsey proved his theorem nearly a century ago, our understanding of the numbers $r(H)$ is still rather limited. Even for the basic case of identical cliques, the bounds ${\sqrt {2}}^r \leq r(K_r) \leq 4^r$ have remained almost unchanged since 1947 (more precisely, there have been no improvements to the exponential constants $\sqrt 2$ and $4$). 

One possible approach to improving the upper bound on $r(K_r)$ is as follows. Fix some $k<r$ and suppose we are given an edge coloring\footnote{For brevity, if not specified, we henceforth use the word ``coloring'' to refer to a red/blue edge coloring of a complete graph.} of $K_N$ with no monochromatic copy of $K_r$. Suppose some blue $K_k$ has at least $n=r(K_r,K_{r-k})$ extensions to a monochromatic $K_{k+1}$. Then, by the definition of $n$, among these $n$ vertices, we must find either a red $K_r$ or a blue $K_{r-k}$, which can be combined with our original blue $K_k$ to yield a blue $K_r$. This contradicts our assumption that the coloring has no monochromatic $K_r$ and, therefore, every monochromatic $K_k$ must have fewer than $n$ monochromatic extensions to a $K_{k+1}$. Equivalently, if we define the \emph{book graph} $B_n\up k$ to consist of $n$ copies of $K_{k+1}$ joined along a common $K_k$ (called the \emph{spine} of the book), then we have  shown that a coloring with no monochromatic $K_r$ must also not contain a monochromatic $B_n \up k$, where $n=r(K_r,K_{r-k})$. In other words,
\[
	r(K_r) \leq r(B_n \up k).
\]
Therefore, one could hope to improve the upper bound on $r(K_r)$ by finding good upper bounds on $r(B_n \up k)$. Such an approach, combined with other techniques coming from the theory of quasirandom graphs, was used by the first author \cite{Conlon2009} to obtain the first superpolynomial improvement to the upper bound of $4^r$.

These observations suggest that one should study the Ramsey numbers $r(B_n \up k)$. In fact, the study of $r(B_n \up k)$ implicitly goes back to Ramsey's original paper \cite{Ramsey}, where he proved the finiteness of $r(K_r)$ by inductively proving the finiteness of $r(B_n \up k)$. However, the modern study of book Ramsey numbers was initiated about four decades ago by Erd\H os, Faudree, Rousseau, and Schelp~\cite{ErFaRoSc} and by Thomason \cite{Thomason82}. Both papers prove a lower bound of the form
\[
	r(B_n \up k) \geq 2^k n-o_k(n),
\]
which follows by considering a uniformly random coloring of $K_N$; in such a coloring, a fixed monochromatic $K_k$ has, in expectation, $2^{-k}(N-k)$ monochromatic extensions and the desired bound follows from applying the Chernoff bound and then the union bound. Alternatively, one can check (e.g.,\ \cite[Theorem 6(i)]{Thomason82}) that a Paley graph of order $q=2^k(n-\Omega_k(\sqrt n))$ contains no $B_n \up k$, yielding a lower bound of the same form. 

Erd\H os, Faudree, Rousseau, and Schelp \cite{ErFaRoSc} asked whether this lower bound or a simple upper bound of the form $4^k n$ is asymptotically tight, while Thomason \cite{Thomason82} conjectured that the lower bound should be asymptotically correct. In fact, he made the stronger conjecture that, for all $n$ and $k$,
\[
	r(B_n \up k) \leq 2^k(n+k-2)+2.
\]
If true, this would yield a huge improvement on the upper bound for $r(K_r)$. Indeed, $B_1^{(r-1)} = K_r$, so we would immediately have $r(K_r) \leq r 2^{r-1}$. In fact, if one could prove that $r(B_n \up k)$ asymptotically matches the lower bound given by the random construction \emph{and} one had sufficiently strong control on the error term in this asymptotic, one might hope for an exponential improvement on $r(K_r)$. The first part of this plan was carried out by the first author~\cite{Conlon}, who answered the question of Erd\H os, Faudree, Rousseau, and Schelp and proved an approximate version of Thomason's conjecture. 
\begin{thm}[Conlon \cite{Conlon}]\label{thm:conlon}
	For any $k \geq 1$,
	\[
		r(B_n \up k)= 2^k n+o_k(n).
	\]
\end{thm}
Unfortunately, the error term $o_k(n)$ decays extremely slowly. More specifically, to obtain the upper bound $2^k n+\varepsilon n$ for some $\varepsilon>0$, the argument in \cite{Conlon} requires $n$ to be at least a tower of twos whose height is a function of $k$ and $1/\varepsilon$. The first author raised the natural question of whether such a dependence is necessary. Our first main result shows that it is not.
\begin{thm}\label{thm:tripleexp}
	For any $k \geq 3$,
	\[
		r(B_n \up k) 
		= 2^k n+O_k\left(\frac{ n}{(\log \log \log n)^{1/25}}\right).
	\]
\end{thm}
\noindent That is, if one wishes to obtain the upper bound $2^k n+\varepsilon n$, then one ``only'' needs $n$ to be triple exponential in $1/\varepsilon$. While this eliminates the tower-type dependence of Theorem \ref{thm:conlon}, it is still far from strong enough to give an exponential improvement to $r(K_r)$ via the approach outlined above.

A second major direction of research in graph Ramsey theory regards the structure of Ramsey colorings, that is, colorings of $K_N$ with no monochromatic $K_r$, where $N$ is ``close'' to the Ramsey number $r(K_r)$. More specifically, we say that an edge coloring of $K_N$ is \emph{$C$-Ramsey} if it contains no monochromatic $K_r$ with $r \geq C \log N$. Since our best lower bounds for $r(K_r)$ come from random constructions, there have been many attempts to show that such $C$-Ramsey colorings exhibit properties that are typical for random colorings. For instance, Erd\H os and Szemer\'edi \cite{ErSzem} proved that such colorings must have both red and blue densities bounded away from $0$; Pr\"omel and R\"odl \cite{PrRo} proved that both the red and blue graphs contain induced copies of all ``small'' graphs (see also \cite{FoSu} for a simpler proof with better bounds); Jenssen, Keevash, Long, and Yepremyan~\cite{JKLY} proved that both the red and blue graphs contain induced subgraphs exhibiting vertices with $\Omega(N^{2/3})$ distinct degrees; and Kwan and Sudakov \cite{KwSu} proved that both the red and blue graphs contain $\Omega(N^{5/2})$ induced subgraphs with distinct numbers of vertices or edges.

Following Chung, Graham, and Wilson~\cite{ChGrWi}, themselves building on work of Thomason~\cite{Thomason87}, we say that an edge coloring of $K_N$ is  \emph{$\theta$-quasirandom} if, for any pair of disjoint vertex sets $X$ and $Y$,
\[
	\left|e_B(X,Y)- \frac 12 |X||Y| \right| \leq \theta N^2,
\]
where $e_B(X,Y)$ denotes the number of blue edges between $X$ and $Y$. Note that since the colors are complementary, we could just as well have used red edges. The importance of this definition is that, for $\theta$ sufficiently small, it implies that the graph has other natural random-like properties, such as that of containing roughly the ``correct'' number of monochromatic copies of all graphs of any fixed order. In light of the research described above, it is natural to ask whether Ramsey colorings are quasirandom. Unfortunately, this is not the case, as may easily be seen by considering the disjoint union of two copies of a $C$-Ramsey coloring and making all edges between them blue, as this coloring is $2C$-Ramsey and is not quasirandom. Nevertheless, S\'os~\cite{Sos} conjectured that true extremal colorings are quasirandom. More precisely, she conjectured that for every $\theta>0$ there is some $r_0$ such that, for all $r \geq r_0$, all edge colorings on $r(K_r)-1$ vertices with no monochromatic $K_r$ are $\theta$-quasirandom. 

A proof of S\'os's conjecture seems completely out of reach at present, if only because it would seem to require an asymptotic determination of the Ramsey number $r(K_r)$. However, an analogous conjecture for book Ramsey numbers was made by Nikiforov, Rousseau, and Schelp \cite{NiRoSc} and, given that we now understand the asymptotic behavior of $r(B_n \up k)$ for all fixed $k$, we might hope that this conjecture is within range. For $k = 2$, where the asymptotic behavior has been long known~\cite{RoSh}, the conjecture was proved by Nikiforov, Rousseau, and Schelp \cite{NiRoSc} themselves. Our second main result establishes their conjecture in full generality. 

\begin{thm}\label{thm:quasirandomness}
	For any $k \geq 2$ and any $0<\theta<\frac 12$, there is some $c=c(\theta,k)>0$ such that if a $2$-coloring of $K_N$ is {not} $\theta$-quasirandom for $N$ sufficiently large, then it contains a monochromatic $B_n\up k$ with $n=(2^{-k}+c)N$. 
\end{thm}

Theorem \ref{thm:quasirandomness} will follow from the following stronger result, which says that in a non-quasirandom coloring, a constant fraction of the monochromatic $K_k$ form the spine of one of these large books.

\begin{thm}\label{thm:strongquasirandom}
	For any $k \geq 2$ and any $0<\theta<\frac 12$, there is some $c_1=c_1(\theta,k)>0$ such that if a $2$-coloring of $K_N$ is not $\theta$-quasirandom for $N$ sufficiently large, then it contains at least $c_1 N^k$ monochromatic $K_k$, each of which has at least $(2^{-k}+c_1)N$ extensions to a monochromatic $K_{k+1}$.
\end{thm}
Moreover, in Theorem \ref{thm:quasirandomconverse}, we prove a converse to this result, which, when combined with Theorem \ref{thm:strongquasirandom}, implies the following result.

\begin{thm}\label{thm:combined-quasirandom}
	Fix $k \geq 2$. A $2$-coloring of the edges of $K_N$ is $o(1)$-quasirandom if and only if all but $o(N^k)$ monochromatic $K_k$ have at most $(2^{-k}+o(1))N$ extensions to a monochromatic $K_{k+1}$. 
\end{thm}

This adds to the long list of properties known to be equivalent to quasirandomness. It also has an interesting consequence related to a famous (and famously false) conjecture of Erd\H os~\cite{Erdos62}. He conjectured that, for any fixed $k\geq 3$, every red/blue coloring of $K_N$ contains at least $(1-o(1))2^{1-\binom k2}\binom Nk$ monochromatic $K_k$, that is, that a uniformly random coloring asymptotically minimizes the number of monochromatic $K_k$. While this conjecture is a simple consequence of Goodman's formula~\cite{Goodman} when $k=3$, Thomason \cite{Thomason89} showed that it is false for all $k \geq 4$. Any coloring witnessing the failure of this conjecture (i.e.,\ with asymptotically fewer than $2^{1-\binom k2}\binom Nk$ monochromatic $K_k$) must not be $o(1)$-quasirandom, as a quasirandom coloring has the same count of $K_k$ as a uniformly random coloring. Therefore, Theorem \ref{thm:combined-quasirandom} implies that any coloring with ``too few'' monochromatic $K_k$ must have the property that a positive proportion of its $K_{k-1}$ lie in ``too many'' monochromatic $K_k$.
In other words, it is impossible to have asymptotically fewer monochromatic $K_k$ than in a random coloring unless these $K_k$ are somehow more clustered than in a random coloring.

The rest of the paper is organized as follows. In Section \ref{sec:regularitytools}, we collect some fairly standard results related to Szemer\'edi's regularity lemma which will be important in our proofs. In Section~\ref{sec:simplifiedproof}, we present a streamlined proof of Theorem \ref{thm:conlon}. Most of the ideas in this proof are already present in~\cite{Conlon}, but the presentation here is simpler. Moreover, various ideas and results from Section \ref{sec:simplifiedproof} will be adapted and reused later in the paper. In Section \ref{sec:basicresult}, we prove Theorem \ref{thm:tripleexp}, which improves the error term in Theorem \ref{thm:conlon}. We then present our quasirandomness results in Section \ref{sec:quasirandomness}, including the proof of Theorem \ref{thm:strongquasirandom} and its converse. We conclude with some further remarks, though there is also an appendix where we consign the proofs of certain technical lemmas.

\subsection{Notation and terminology}
If $X$ and $Y$ are two vertex subsets of a graph, let $e(X,Y)$ denote the number of pairs in $X \times Y$ that are edges. We will often normalize this and consider the \emph{edge density},
\[
	d(X,Y)=\frac{e(X,Y)}{|X||Y|}.
\]
If we consider a red/blue coloring of the edges of a graph, then $e_B(X,Y)$ and $e_R(X,Y)$ will denote the number of pairs in $X \times Y$ that are blue and red edges, respectively. Similarly, $d_B$ and $d_R$ will denote the blue and red edge densities, respectively. Finally, for a vertex $v$ and a set $Y$, we will sometimes abuse notation and write $d(v,Y)$ for $d(\{v\},Y)$ and similarly for $d_B$ and $d_R$. 

An \emph{equitable partition} of a graph $G$ is a partition of the vertex set $V(G)=V_1 \sqcup \dotsb \sqcup V_m$ with $||V_i|-|V_j||\leq 1$ for all $1 \leq i,j \leq m$. A pair of vertex subsets $(X,Y)$ is said to be \emph{$\varepsilon$-regular} if, for every $X' \subseteq X$, $Y' \subseteq Y$ with $|X'| \geq \varepsilon|X|$, $|Y'| \geq \varepsilon |Y|$, we have
\[
	|d(X,Y)-d(X',Y')| \leq \varepsilon.
\]
Note that we do not require $X$ and $Y$ to be disjoint. In particular, we say that a single vertex subset $X$ is \emph{$\varepsilon$-regular} if the pair $(X,X)$ is $\varepsilon$-regular. We will often need a simple fact, known as the \emph{hereditary property} of regularity, which asserts that for any $0<\alpha\leq\frac 12$, if $(X,Y)$ is $\varepsilon$-regular and $X' \subseteq X$, $Y' \subseteq Y$ satisfy $|X'| \geq \alpha |X|$, $|Y'| \geq \alpha |Y|$, then $(X',Y')$ is $(\varepsilon/\alpha)$-regular. 

\begin{numrem}\label{rem:equitable-partitions}
All logarithms are to base $2$ unless otherwise stated. For the sake of clarity of presentation, we systematically omit floor and ceiling signs whenever they are not crucial. In this vein, whenever we have an equitable partition of a vertex set, we will always assume that all of the parts have exactly the same size, rather than being off by at most one. Because the number of vertices in our graphs will always be ``sufficiently large'', this has no effect on our final results.
\end{numrem}

\section{Some regularity tools}\label{sec:regularitytools}
The main regularity result we will need is the following, which is a slight strengthening of the usual version of Szemer\'edi's regularity lemma. 
\begin{lem}\label{reglem}
	For every $\varepsilon>0$ and $M_0 \in \mathbb N$, there is some $M=M(\varepsilon,M_0)>M_0$ such that for every graph $G$, 
	there is an equitable partition $V(G)=V_1 \sqcup \dotsb \sqcup V_m$ into $M_0 \leq m \leq M$ parts so that the following hold:
	\begin{enumerate}
		\item Each part $V_i$ is $\varepsilon$-regular and
		\item For every $1 \leq i \leq m$, there are at most $\varepsilon m$ values $1 \leq j \leq m$ such that the pair $(V_i,V_j)$ is not $\varepsilon$-regular.
	\end{enumerate}
\end{lem}
\noindent 
Note that this strengthens Szemer\'edi's regularity lemma in two ways: first, it ensures that every part of the partition is $\varepsilon$-regular with itself and, second, it imposes some structure on the fewer than $\varepsilon m^2/2$ irregular pairs, ensuring that they are reasonably well-distributed. 

In order to prove Lemma \ref{reglem}, we will need some other results. The first asserts that every graph contains a reasonably large $\varepsilon$-regular subset.
\begin{lem}[Conlon--Fox {\cite[Lemma 5.2]{CoFo}}]\label{conlonfox}
	Given $0<\varepsilon<\frac 12$, let $\delta=2^{-\varepsilon^{-(10/\varepsilon)^4}}$. Then, for every graph $G$, there is an $\varepsilon$-regular subset $W \subseteq V(G)$ with $|W| \geq \delta |V(G)|$.
\end{lem}
The next lemma asserts that inside an $\varepsilon$-regular set of vertices, we may find a subset of any specified cardinality whose regularity is not much worse, provided we do not restrict to too small a cardinality. The proof of this lemma may be found in the appendix.
\begin{lem}\label{lem:random-subgraph-reg}
	Fix $0<\varepsilon<\frac 15$ and let $t \geq \varepsilon^{-4}$ be an integer. Let $G$ be a graph on at least $t$ vertices and suppose that $V(G)$ is $\varepsilon$-regular. Then there is a subset $U \subseteq V(G)$ with $|U|=t$ such that $U$ is $(10 \varepsilon)^{1/3}$-regular. In fact, a randomly chosen $U \in \binom {V(G)}t$ will be $(10 \varepsilon)^{1/3}$-regular with probability tending to $1$ as $\varepsilon\to 0$.
\end{lem}

Using the previous two lemmas, we can prove that any graph may be equitably partitioned into $\varepsilon$-regular subsets. We will use this result instead of Lemma \ref{reglem} as the main partitioning lemma in the proof of Theorem \ref{thm:tripleexp}.
\begin{lem} \label{lem:partition}
	Fix $0<\varepsilon< \frac 1{100}$ and suppose that $G$ is a graph on $n \geq 2^{1/\varepsilon^{(10/\varepsilon)^{15}}}$ vertices. Then $G$ has an equitable partition $V(G)=V_1 \sqcup \dotsb \sqcup V_K$ such that each $V_i$ is $\varepsilon$-regular, where $K=K(\varepsilon)$ is a constant depending only on $\varepsilon$ satisfying $2^{1/\varepsilon^{(10/\varepsilon)^{12}}} \leq K(\varepsilon) \leq 2^{1/\varepsilon^{(10/\varepsilon)^{15}}}$.
\end{lem}
\begin{rem}
	Unlike Szemer\'edi's regularity lemma, Lemma \ref{lem:partition} makes no assertion about regularity between the parts, only that they are all $\varepsilon$-regular with themselves. Lemma \ref{lem:partition} is very similar to \cite[Lemma 5.7]{CoFo}, which is also proven by repeatedly applying Lemma \ref{conlonfox}. However, our lemma is stronger both in guaranteeing that the partition is equitable and in having the number of parts be a fixed constant depending only on $\varepsilon$, rather than lying in some range. 
\end{rem}
\begin{proof}[Proof of Lemma \ref{lem:partition}]
	Let $n=|V(G)|$. Let $\varepsilon_0=\varepsilon^3/10^4$, let $\delta_0=\delta(\varepsilon_0)=2^{-\varepsilon_0^{-(10/\varepsilon_0)^4}}$ be the parameter from Lemma \ref{conlonfox}, and set $\delta_1=\varepsilon^2 \delta_0/10$. Note that by our assumption on $n$, we have that $n \geq \varepsilon_0^{-4} \delta_1^{-1}$.

	We will iteratively construct a sequence of disjoint $(\varepsilon/10)$-regular vertex subsets $U_1,U_2,\ldots$ with $|U_i|=\delta_1 n$ for all $i$. To begin the sequence, we apply Lemma \ref{conlonfox} to find a set $W_1 \subseteq V(G)$ with $|W_1| \geq \delta_0 n$ that is $\varepsilon_0$-regular. We now apply Lemma \ref{lem:random-subgraph-reg} with $t=\delta_1 n$ to $W_1$ to find an $(\varepsilon/10)$-regular subset $U_1 \subseteq W_1$ with $|U_1|=\delta_1 n$. Note that we may apply Lemma \ref{lem:random-subgraph-reg} since, by our assumption on $n$, we have that $t \geq \varepsilon_0^{-4}$. 

	Suppose now that we have defined disjoint sets $U_1,\ldots,U_i$ with $i \leq (1- \varepsilon^2/10)/\delta_1$. Let $V_{i+1}=V \setminus (U_1 \cup \dotsb \cup U_i)$. Then we apply Lemma \ref{conlonfox} to $V_{i+1}$ to find an $\varepsilon_0$-regular subset $W_{i+1} \subseteq V_{i+1}$ with
	\[
		|W_{i+1}| \geq \delta_0 |V_{i+1}|=\delta_0 (1-i \delta_1)n \geq \delta_0\left(1- \frac{1- \varepsilon^2/10}{\delta_1}\delta_1\right)n =\frac{\varepsilon^2 \delta_0 n}{10}=\delta_1 n.
	\]
	Therefore, we may apply Lemma \ref{lem:random-subgraph-reg} to $W_{i+1}$ to find an $(\varepsilon/10)$-regular subset $U_{i+1} \subseteq W_{i+1}$ with $|U_{i+1}|=\delta_1 n$, so continuing the sequence. 

	This process stops once we have $K := \lfloor (1- \varepsilon^2/10)/\delta_1 \rfloor + 1$ sets $U_1,  U_2, \dots, U_K$. At that point, we will have placed at least a $(1- \varepsilon^2/10)$-fraction of the vertices into one of the sets $U_1,\ldots,U_K$. The remaining vertices we arbitrarily and equitably partition into sets $Z_1,\ldots,Z_K$, where
	\[
		|Z_i|\leq\frac{\varepsilon^2 n/10}{K}\leq\frac{\varepsilon^2/10}{1- \varepsilon^2/10} (\delta_1 n) <\frac{\varepsilon^2}9 |U_i|.
	\]
	Finally, we set $V_i=U_i \cup Z_i$ to obtain an equitable partition of $V(G)$. Lemma 5.6 in \cite{CoFo} shows that if $(X,Y)$ is an $\alpha$-regular pair of vertices and $Z$ is a set of vertices disjoint from $Y$ with $|Z| \leq \beta |Y|$, then $(X,Y \cup Z)$ is $(\alpha+\beta+\sqrt \beta)$-regular. Applying this fact to $(U_i, U_i)$ twice with $\beta=\varepsilon^2/9$ shows that $V_i$ is $(\varepsilon/10+2(\varepsilon^2/9)+2(\varepsilon/3))$-regular and thus $\varepsilon$-regular. 
\end{proof}

\begin{proof}[Proof of Lemma \ref{reglem}]
	Let $\varepsilon_1=\varepsilon/2$ and $\varepsilon_2=\varepsilon^2/128$ and let $K_1=K(\varepsilon_1)\leq 2^{1/\varepsilon_1^{(10/\varepsilon_1)^{15}}}$ be the parameter from Lemma \ref{lem:partition}. Finally, let $\eta=\min\{\varepsilon_1/K_1, \varepsilon \cdot\varepsilon_2/2\}>0$. The usual form of Szemer\'edi's regularity lemma (e.g.,\ \cite[Theorem 2]{KoShSiSz}) says that there is some $L=L(\eta,M_0)>M_0$ such that we can find an equitable partition $V(G)=W_1 \sqcup \dotsb \sqcup W_\ell$ with $\max\{M_0,1/\eta\} \leq \ell \leq L$ where all but at most $\eta \binom \ell 2$ pairs of parts $(W_i,W_j)$ are $\eta$-regular. We now apply Lemma \ref{lem:partition} to each $W_i$ to get an equitable partition $W_i=U_{i1} \sqcup \dotsb \sqcup U_{iK_1}$ such that each part is $\varepsilon_1$-regular. Since the $W_i$ formed an equitable partition and each $W_i$ is cut up into the same number $K_1$ of parts, the resulting partition of $V(G)$ is equitable. Moreover, since each $U_{ij}$ is at least a $1/K_1$-fraction subset of $W_i$, the hereditary property of regularity implies that if $(W_{i_1},W_{i_2})$ is $\eta$-regular, then $(U_{i_1 j_1},U_{i_2 j_2})$ is $\eta K_1$-regular for all $j_1,j_2$. Therefore, all but an $\eta$-fraction of the pairs $(U_{i_1j_1},U_{i_2j_2})$ with $i_1 \neq i_2$ are $\eta K_1$-regular. By our choice of $\eta$, we know that $\eta K_1\leq \varepsilon_1$. So we have found an equitable partition where each part is $\varepsilon_1$-regular and all but an $\eta$-fraction of the pairs $(U_{i_1 j_1},U_{i_2 j_2})$ with $i_1 \neq i_2$ are $\varepsilon_1$-regular. Moreover, the fraction of pairs $(U_{i_1 j_1}, U_{i_2 j_2})$ with $i_1 =i_2$ is $1/\ell$, so we see that the total fraction of irregular pairs is at most $\eta + 1/\ell \leq 2 \eta$.

	We will now rename the parts as $U_1,\ldots,U_m$, where $m=K_1\ell$, since we no longer need to track which $W$ part each $U$ part came from. We are almost done, except that the irregular pairs might still be badly distributed: some $U_i$ might be involved in more than $\varepsilon m$ irregular pairs. However, since there are at most $2\eta \binom m2$ irregular pairs, the number of such ``bad'' $U_i$ is at most $(2\eta/\varepsilon)m  \leq \varepsilon_2 m$. Therefore, at most an $\varepsilon_2$-fraction of the vertices are contained in a bad $U_i$. We now equitably, but otherwise arbitrarily, distribute these vertices into the remaining at least $(1- \varepsilon_2)m$ parts to obtain a new partition $V_1,\ldots,V_{m'}$, where $V_i$ is obtained from $U_i$ by adding to it at most $\beta|U_i|$ vertices, where $\beta=\frac{\varepsilon_2}{1- \varepsilon_2}<2\varepsilon_2$. We again apply Lemma 5.6 from \cite{CoFo}. In fact, we will only need a slightly weaker bound, namely, that if $(X,Y)$ is an $\alpha$-regular pair of vertices and $Z$ is a set of vertices disjoint from $Y$ with $|Z| \leq \beta |Y|$, then $(X,Y \cup Z)$ is $(\alpha+2\sqrt \beta)$-regular. Therefore, by applying this fact twice, we see that if $(U_i,U_j)$ was an $\varepsilon_1$-regular pair of good parts, then $(V_i,V_j)$ is $(\varepsilon_1+4\sqrt \beta)$-regular. Moreover,
	\[
		\varepsilon_1+4\sqrt \beta < \varepsilon_1+4\sqrt{2 \varepsilon_2}=\frac{\varepsilon}2+4\sqrt{\frac {\varepsilon^2}{64}}=\varepsilon.
	\]
	By the exact same computation, we see that each $V_i$ is $\varepsilon$-regular, since each $U_i$ was $\varepsilon_1$-regular. Therefore, $V_1,\ldots,V_{m'}$ is the desired partition. 
\end{proof}

Another important tool will be a standard counting lemma (see, e.g.,\ \cite[Theorem~3.27]{Zhao}).

\begin{lem}\label{countinglemma}
	Suppose that $V_1,\ldots,V_k$ are (not necessarily distinct) subsets of a graph $G$ such that all pairs $(V_i,V_j)$ are $\varepsilon$-regular. Then the number of labeled copies of $K_k$ whose $i$th vertex is in $V_i$ for all $i$ is 
	\[
		\left( \prod_{1 \leq i<j\leq k}d(V_i,V_j) \pm \varepsilon \binom k2 \right) \prod_{i=1}^k |V_i|.
	\]
\end{lem}

We will frequently use the following consequence of Lemma~\ref{countinglemma}, designed to count monochromatic extensions of cliques and thus estimate the size of monochromatic books.

\begin{cor}\label{cor:randomclique}
	Let $\eta,\delta \in (0,1)$ be parameters with $\eta \leq \delta^3/k^2$. Suppose $U_1,\ldots,U_k$ are (not necessarily distinct) vertex sets in a graph $G$ and all pairs $(U_i,U_j)$ are $\eta$-regular with $\prod_{1 \leq i<j\leq k}d(U_i,U_j) \geq \delta$. Let $Q$ be a randomly chosen copy of $K_k$ with one vertex in each $U_i$ with $1 \leq i \leq k$ and say that a vertex $u$ extends $Q$ if $u$ is adjacent to every vertex of $Q$. Then, for any $u$,
	\begin{equation}\label{eq:randomclique}
		\pr(u \text{ extends }Q) \geq \prod_{i=1}^k d(u,U_i)-4 \delta.
	\end{equation}
\end{cor}

\begin{proof}
	Since the right-hand side of (\ref{eq:randomclique}) is negative if $d(u,U_i) \leq 4\delta$ for some $i$, the conclusion is vacuously true in this case. Thus, we may assume that $d(u,U_i) > 4\delta$ for all $i$. 

	First, by the counting lemma, Lemma \ref{countinglemma}, the number of copies of $K_k$ with one vertex in each $U_i$ is at most
	\[
		\left(\prod_{1 \leq i<j \leq k} d(U_i,U_j)+\eta \binom k2\right)\prod_{i=1}^k |U_i|.
	\]
	On the other hand, for a vertex $u$, let $U_i'$ be its neighborhood in $U_i$, so that $|U_i'|=d(u,U_i)|U_i|\geq \delta |U_i|$. Then, by the hereditary property of regularity, we know that each pair $(U_i',U_j')$ is $\frac \eta \delta$-regular. Therefore, by Lemma \ref{countinglemma}, we know that the number of $K_k$ with one vertex in each $U_i'$ is at least
	\[
		\left(\prod_{1 \leq i<j \leq k} d(U_i',U_j')-\frac \eta \delta \binom k2\right)\prod_{i=1}^k |U_i'|.
	\]
	Note that since $(U_i, U_j)$ is $\eta$-regular and $\delta > \eta$, we also know that $d(U_i',U_j') \geq d(U_i,U_j)- \eta$
	and, therefore,
	\[
		\prod_{1 \leq i <j \leq k} d(U_i',U_j') \geq \prod_{1 \leq i<j \leq k} d(U_i,U_j)-\eta \binom k2 \geq \prod_{1 \leq i<j \leq k} d(U_i,U_j)-\frac \eta \delta \binom k2.
	\]
	Putting this together, we find that the number of $K_k$ with one vertex in each $U_i'$ is at least
	\[
		\left( \prod_{1 \leq i<j \leq k} d(U_i,U_j)-2 \frac \eta \delta \binom k2 \right) \prod_{i=1}^k d(u,U_i)|U_i|.
	\]
	Now, the probability that $u$ extends $Q$ is precisely the probability that $Q$ has one vertex in each $U_i'$. Therefore, dividing the number of such cliques by the total number of cliques with one vertex in each $U_i$ gives us the probability that $u$ extends $Q$. By the calculations above, we get
	\begin{align}
		\pr(u \text{ extends }Q) &\geq \frac{\left(\prod_{1 \leq i<j \leq k} d(U_i,U_j)-2\frac \eta \delta \binom k2\right)\prod_{i=1}^k (d(u,U_i)|U_i|)}{\left(\prod_{1 \leq i<j \leq k} d(U_i,U_j)+\eta \binom k2\right)\prod_{i=1}^k |U_i|} \notag\\
		&= \frac{\prod_{1 \leq i<j \leq k} d(U_i,U_j)-2\frac \eta \delta \binom k2}{\prod_{1 \leq i<j \leq k} d(U_i,U_j)+\eta \binom k2} \prod_{i=1}^k d(u,U_i) \notag \\
		& \geq \frac{\delta- 2 \frac \eta \delta \binom k2}{\delta+\eta \binom k2}\prod_{i=1}^k d(u,U_i)\label{eq:monotonicity}\\
		&\geq (1-4 \delta) \prod_{i=1}^k d(u,U_i)\label{eq:dividing-trick} \\
		& \geq \prod_{i=1}^k d(u,U_i)- 4 \delta \label{eq:one-minus-trick}.
	\end{align}
	In (\ref{eq:monotonicity}), we used that the function $(x-y)/(x+z)$ is monotonically increasing in $x$ for all $y,z>0$, as well as the assumption that $\prod d(U_i,U_j)\geq \delta$. In (\ref{eq:dividing-trick}), we used that $(x-2y)/(x(1+y)) \geq 1-2x$ for all positive $x,y$ with $y<x^2/2$, applying this with $x=\delta$ and $y = \frac \eta \delta \binom k2$, where the bound $y<x^2/2$ holds by our assumption that $\eta \leq \delta^3/k^2 <\delta^3/2\binom k2$. Finally, in (\ref{eq:one-minus-trick}), we used that $(1-x)y \geq y-x$ for all $y \in [0,1]$.
\end{proof}

\section{A simplified proof of Theorem~\ref{thm:conlon}}\label{sec:simplifiedproof}
In this section, we present another proof of Theorem \ref{thm:conlon} which gives bounds comparable to those obtained in~\cite{Conlon}. Though many of the ideas are the same in both proofs, we believe that the proof here is conceptually simpler than that in~\cite{Conlon}. The main differences are that we use  Lemma~\ref{reglem} instead of the usual regularity lemma and also that we use averaging arguments in a few more places. As a result, we only need to find a clique in the reduced graph, instead of a clique blow-up as in~\cite{Conlon}.

Suppose we are given a red/blue coloring of the edges of $K_N$, where $N=(2^k+\varepsilon)n$ for some $\varepsilon>0$. We wish to find a monochromatic $B_n \up k$ in this coloring. Doing this for all $\varepsilon$ and all sufficiently large $n$ will prove Theorem \ref{thm:conlon}. The key observation, which also implicitly underlies the proof in \cite{Conlon}, is that to find the ``large'' structure of a monochromatic $B_n\up k$, it suffices to find a different ``small'' structure, which we call a \emph{good configuration}.

\begin{Def}\label{def:good-config}
	Fix $k \geq 2$ and let $\eta,\delta>0$ be some parameters. A \emph{$(k,\eta,\delta)$-good configuration} is a collection of $k$ disjoint vertex sets $C_1,\ldots, C_k \subseteq V(K_N)$ with the following properties:
	\begin{enumerate}
		\item \label{cond:red-density-lb} Each $C_i$ is $\eta$-regular with itself and has red density at least $\delta$ and
		\item \label{cond:blue-density-lb} For all $i \neq j$, the pair $(C_i,C_j)$ is $\eta$-regular and has blue density at least $\delta$.
	\end{enumerate}
\end{Def}

\begin{Def}
	A $(k,\eta,\delta)$-good configuration $C_1,\ldots,C_k$ is called a \emph{$(k,\eta,\delta)$-great configuration} if the density conditions in Properties \ref{cond:red-density-lb} and \ref{cond:blue-density-lb} are replaced by the stronger conditions that
	\[
		d_R(C_i)^{\binom k2} \geq \delta \qquad \text{ and } \qquad \prod_{1 \leq i <j \leq k} d_B(C_i,C_j) \geq \delta,
	\]
	where the first condition holds for all $i \in [k]$. Note that we still require the same $\eta$-regularity conditions as in Definition \ref{def:good-config}, while strengthening the density assumptions.
\end{Def}

\begin{rem}
	Note that good and great configurations are equivalent up to a polynomial change in the parameters. Certainly, a $(k,\eta,\delta)$-great configuration is also $(k,\eta,\delta)$-good, for if the product of some numbers in $[0,1]$ is at least $\delta$, then each of these numbers must be at least $\delta$. On the other hand, every $(k,\eta,\delta)$-good configuration is also $(k,\eta,\delta^{\binom k2})$-great. 
\end{rem}

We will first describe a process that finds either a monochromatic $B_n \up k$ or a good configuration and then later see how to use this good configuration to find a monochromatic book. We set $\delta=2^{-4k}\varepsilon$ and $\eta=\delta^{2k^2}$.

We begin by applying Lemma \ref{reglem} to the red graph, with the parameter $\eta$ as above and with $M_0=1/\eta$. We obtain an equitable partition $V(K_N)=V_1 \sqcup \dotsb \sqcup V_m$ with a bounded number of parts such that, for each $i$, $V_i$ is $\eta$-regular and there are at most $\eta m$ values of $j$ such that $(V_i,V_j)$ is not $\eta$-regular. Note that since the colors are complementary, the same holds for the blue graph. Without loss of generality, at least $m' \geq m/2$ of the parts have internal red density at least $\frac 12$. By renaming if necessary, we may assume that $V_1,\ldots,V_{m'}$ are these red parts. We introduce new vertices $v_1,\ldots,v_m$ and form a reduced graph $G$ on the vertex set $v_1,\ldots,v_{m}$ by connecting $v_i$ to $v_j$ (for $i \neq j$) if $(V_i,V_j)$ is $\eta$-regular and $d_B(V_i,V_j) \geq \delta$. Let $G'$ be the subgraph of $G$ induced by the ``red'' vertices $v_i$ with $1 \leq i \leq m'$. Suppose that, in $G'$, some $v_i$ has at least $(2^{1-k}+2\eta)m'$ non-neighbors. Then, since $v_i$ has at most $\eta m \leq 2 \eta m'$ non-neighbors coming from irregular pairs, this means that there are at least $2^{1-k}m'$ parts $V_j$ with $1 \leq j \leq m'$ such that $(V_i,V_j)$ is $\eta$-regular and $d_R(V_i,V_j) \geq 1- \delta$. Let $J$ be the set of all these indices $j$ and let $U=\bigcup_{j \in J}V_j$ be the union of all of these $V_j$. We then have
\begin{equation}\label{eq:density-lb}
	e_R(V_i,U)=\sum_{j \in J} e_R(V_i,V_j) \geq \sum_{j \in J} (1- \delta) |V_i||V_j|=(1- \delta) |V_i||U|.
\end{equation}
Let $V_i' \subseteq V_i$ denote the set of vertices $v \in V_i$ with $e_R(v,U) \geq (1-2 \delta)|U|$. Then we may write
\begin{equation}\label{eq:density-ub}
	e_R(V_i,U)=\sum_{v \in V_i'} e_R(v,U)+\sum_{v \in V_i \setminus V_i'} e_R(v,U) \leq |V_i'||U|+(1-2 \delta)|V_i \setminus V_i'| |U|.
\end{equation}
Combining equations (\ref{eq:density-lb}) and (\ref{eq:density-ub}), we find that $|V_i'| \geq \frac 12 |V_i|$, where every vertex in $V_i'$ has red density at least $1-2 \delta$ into $U$. Moreover, we may apply the $\eta$-regularity of $V_i$ to conclude that the internal red density of $V_i'$ is at least $\frac 12- \eta \geq \delta$, while the hereditary property of regularity implies that $V_i'$ is $2\eta$-regular. By the counting lemma, Lemma \ref{countinglemma}, $V_i'$ contains at least 
\begin{equation*}\label{eq:cliquecount}
	\left( \delta^{\binom k2}- 2\eta \binom k2 \right) |V_i'|^k \geq \frac 12 \delta^{\binom k2} |V_i'|^k>0
\end{equation*}
red copies of $K_k$, where we used that $2 \eta \binom k2<\frac 12 \delta^{\binom k2}$. Fix one such red $K_k$. Since each vertex in this red $K_k$ has at least $(1- 2 \delta)|U|$ red edges into $U$, this red $K_k$ has at least $(1-2k \delta)|U|$ red common neighbors in $U$. Finally, since $U$ contains at least $2^{1-k} m' \geq 2^{-k}m$ parts $V_j$ and the partition is equitable (and, as explained in Remark~\ref{rem:equitable-partitions}, we are assuming equitable always means exactly equitable),
we have that $|U| \geq 2^{-k} N$. Thus, the red $K_k$ we found in $V_i'$ has at least $(1-2k \delta)|U|$ extensions to a red $K_{k+1}$ and
\[
	(1- 2k \delta)|U| \geq (1- 2k \delta)(1+2^{-k}\varepsilon)n\geq n,
\]
where we used that $2^{-k} \varepsilon \geq 4k \delta$ and $(1-x)(1+2x) \geq 1$ for all $x \in [0,\frac 12]$. Thus, this red $K_k$ gives us our desired $B_n \up k$. 

Therefore, we may assume that every vertex in $G'$ has degree at least $(1-2^{1-k}- 2\eta)m'$. By Tur\'an's theorem, as long as
\[
	2^{1-k}+2\eta < \frac{1}{k-1},
\]
$G'$ will contain a $K_k$. But, since $\eta<2^{-k}$, this condition holds for all $k \geq 2$, so $G'$ contains a copy of $K_k$ with vertices $v_{i_1},\ldots,v_{i_k}$. Set $C_j=V_{i_j}$. Then we have found a $(k,\eta,\delta)$-good configuration, since every pair $(C_i,C_j)$ with $i \leq j$ is $\eta$-regular, each $C_i$ has red density at least $\frac 12 \geq \delta$, and $d_B(C_i,C_j) \geq \delta$ for $i \neq j$. 
The following lemma therefore completes the proof. It is stated for great configurations for later convenience, but, as noted above, our $(k,\eta,\delta)$-good configuration is also $(k,\eta,\delta^{\binom k2})$-great. 

\begin{lem}\label{lem:goodconfigsuffices}
	Suppose a red/blue coloring of $K_N$, with $N=(2^k+\varepsilon)n$, contains a $(k,\eta,\delta)$-great configuration with $\delta \leq 2^{-2k-3} \varepsilon$ and $\eta\leq\delta^{3}/k^2$. Then it also contains a monochromatic $B_n \up k$.
\end{lem}

\begin{proof}
	Let $C_1,\ldots,C_k$ be the $(k,\eta,\delta)$-great configuration. First, observe that by the counting lemma, Lemma \ref{countinglemma}, the number of blue $K_k$ with one vertex in each $C_i$ is at least
	\begin{align*}
		\left( \prod_{1 \leq i <j \leq k}d_B(C_i,C_j) -\eta \binom k2 \right) \prod_{i=1}^k |C_i| &\geq \left( \delta- \eta \binom k2 \right) \prod_{i=1}^k |C_i| \geq \left(\delta-\delta^3\right)\prod_{i=1}^k |C_i|>0,
	\end{align*}
	since the definition of a great configuration includes that $\prod d_B(C_i,C_j) \geq \delta$. Therefore, there is at least one blue $K_k$ with one vertex in each $C_i$. Similarly, for any $i$, the number of red $K_k$ inside $C_i$ is at least
	\begin{align*}
		\left( d_R(C_i)^{\binom k2} -\eta \binom k2 \right) |C_i|^k &\geq \left( \delta- \eta \binom k2 \right) |C_i|^k>0.
	\end{align*}
	Thus, every $C_i$ contains at least one red $K_k$. 

	Next we will need an analytic inequality, essentially \cite[Lemma~8]{Conlon}. The proof of this lemma and a stronger, stability version that we will need later may be found in the appendix. 

	\begin{lem}\label{xilemma}
		For any $x_1,\ldots,x_k \in [0,1]$,
		\[
			\prod_{i=1}^k x_i + \frac 1k \sum_{i=1}^k (1-x_i)^k \geq 2^{1-k}.
		\]
	\end{lem}

	Now, for any vertex $v$ and any $i \in [k]$, consider the blue density $x_i(v):=d_B(v,C_i)$. By Lemma \ref{xilemma}, we know that
	\[
		\prod_{i=1}^k x_i(v)+ \frac 1k \sum_{i=1}^k (1-x_i(v))^k \geq 2^{1-k}.
	\]
	Summing this inequality over all $v$, we get that
	\[
		\sum_{v \in V}\prod_{i=1}^k x_i(v) +\frac 1k \sum_{i=1}^k \sum_{v \in V} (1-x_i(v))^k \geq 2^{1-k}N.
	\]
	Since the sum of these two quantities is at least $2^{1-k}N$, one of them must be at least $2^{-k}N$. First, suppose that
	\begin{equation}\label{eq:probability}
		\sum_{v \in V} \prod_{i=1}^k x_i(v) \geq 2^{-k}N.
	\end{equation}
	For a given vertex $v$, if we pick $v_i \in C_i$ with $1 \leq i \leq k$ uniformly and independently at random, then $\prod_{i=1}^k x_i(v)$ is the probability that the edges $(v,v_i)$ are blue. Hence, inequality (\ref{eq:probability}) implies that for a random $v$ and random $v_i \in C_i$, there is a probability at least $2^{-k}$ that all the edges $(v,v_i)$ are blue. Heuristically, this fact, combined with the regularity of the pairs $(C_i,C_j)$, implies that a random blue $K_k$ spanned by $(C_1,\ldots,C_k)$ will also have probability close to $2^{-k}$ of being in the blue neighborhood of a random $v$. More formally, by applying Corollary \ref{cor:randomclique} for each $v$ and summing, we see that the expected number of blue extensions of a randomly chosen blue $K_k$ spanned by $(C_1,\ldots,C_k)$ is at least
	\begin{align*}
		\sum_{v \in V} \left( \prod_{i=1}^k x_i(v)-4 \delta \right)&\geq (2^{-k}-4 \delta)N
		=(2^{-k}-4 \delta)(2^k+\varepsilon)n
		\geq(1+2^{-k} \varepsilon-2^{k+3}\delta)n
		\geq n,
	\end{align*}
	by our choice of $\delta \leq 2^{-2k-3}\varepsilon$. Therefore, there must exist some blue $K_k$ with at least $n$ blue extensions, giving us our desired blue $B_n \up k$.

	On the other hand, suppose that 
	\[
		\frac 1k \sum_{i=1}^k \sum_{v \in V} (1-x_i(v))^k \geq 2^{-k}N.
	\]
	Then there must exist some $1 \leq i\leq k$ for which $\sum_{v \in V} (1-x_i(v))^k \geq 2^{-k} N$.
	For this $i$, similar logic applies: this fact, together with the regularity of $C_i$, implies that for a random red $K_k$ in $C_i$ and for a random $v \in V$, $v$ will form a red extension of the $K_k$ with probability close to $2^{-k}$. More precisely, by Corollary \ref{cor:randomclique}, the expected number of extensions of a random red $K_k$ in $C_i$ is at least\footnote{Strictly speaking, if $v \in C_i$, then $d_R(v,C_i) \neq 1-x_i(v)$,
	as $v$ has no edge to itself. However, this tiny loss can be absorbed into the error terms 
	and the result does not change.} 
	\[
		\sum_{v \in V} \left((1-x_i(v))^k-4 \delta\right) \geq (2^{-k}-4 \delta)N \geq n,
	\]
	by the same computation as above. Therefore, we see that a randomly chosen red $K_k$ inside $C_i$ will have at least $n$ red extensions in expectation. Hence, there must exist a red $B_n \up k$, completing the proof.
\end{proof}

Recall that previously we found a $(k,\eta,\delta)$-good configuration with $\delta=2^{-4k}\varepsilon$ and $\eta=\delta^{2k^2}$. This is also a $(k,\eta,\delta')$-great configuration, where $\delta'=\delta^{\binom k2}$. Therefore, we can apply Lemma \ref{lem:goodconfigsuffices}, since $\delta' \leq 2^{-2k-3} \varepsilon$ and $\eta=\delta^{2k^2} \leq (\delta')^{4} \leq (\delta')^3/k^2$. Applying Lemma \ref{lem:goodconfigsuffices} yields our desired monochromatic $B_n \up k$ and completes the proof of Theorem \ref{thm:conlon}.

\section{A new proof, with better bounds}\label{sec:basicresult}

In this section, we prove Theorem~\ref{thm:tripleexp}, which we now restate in the following more precise form.

\begin{thm:tripleexp}
	Fix $k \geq 3$. Then, for any $n \geq 2^{2^{2^{k^{25k^2}2^{100 k^3}}}}$,
	\[
		r(B_n \up k) \leq 2^k n+\frac{n}{(\log \log \log n)^{1/25}}.
	\]
\end{thm:tripleexp}

The proof of this result follows similar lines to the proof of Theorem \ref{thm:conlon} presented in Section~\ref{sec:simplifiedproof}, except that we must now avoid invoking Szemer\'edi's regularity lemma at all costs, as doing so would necessarily result in tower-type bounds. Instead, we will invoke Lemma \ref{lem:partition} to obtain a somewhat structured partition of our vertex set (which results in only a double-exponential loss in our parameters) and then attempt to locate a great configuration in this partition. If we are able to do so, then Lemma \ref{lem:goodconfigsuffices} guarantees us our desired monochromatic book. Assuming that at least half the parts in the partition are red, we first show that such a great configuration exists unless there are very few blue $K_k$ between the red parts of our partition and then, in Section \ref{fewbluekks}, we show how to guarantee a monochromatic book also in that case (without finding a great configuration). The rest of the section spells out the details of this approach. We will assume throughout that $\varepsilon\leq 2^{-4k^3}/k^{k^2}$ and set $\delta=2^{-2k-3} \varepsilon$, $\zeta=\delta^3/k^2$, and $\eta=\zeta^{2k^2 \zeta^{-5}}$. 

Suppose we are given a red/blue coloring of $K_N$, where $N=(2^k+\varepsilon)n$, and we wish to find a monochromatic $B_n \up k$. We apply Lemma \ref{lem:partition} to the red graph with parameter $\eta$ as above to obtain an equitable partition of the vertices $V=V_1 \sqcup \dotsb \sqcup V_m$ where each part is $\eta$-regular in red and $m=K(\eta)$; to do this, we assume that $N \geq 2^{1/\eta^{(10/\eta)^{15}}}$. Since the colors are complementary, each part is also $\eta$-regular in blue. We call a part $V_i$ \emph{red} if at least half its edges are colored red and \emph{blue} otherwise. Without loss of generality, we may assume that at least $m/2$ of the parts are red. Let $R$ be the set of all vertices in red parts, that is,
\begin{equation}
	R=\bigcup_{i:V_i \text{ red}} V_i.\label{eq:def-of-R}
\end{equation}

The main tool that we will use to find a great configuration in $R$ is a weak regularity lemma due to Duke, Lefmann, and R\"odl, stated below. We need the following terminology. A \emph{cylinder} is simply a product set, that is, a set of the form $S=S_1 \times \dotsb \times S_k$. We say that such a cylinder of vertex sets is \emph{$\varepsilon$-regular} if every pair $(S_i,S_j)$ with $1 \leq i \neq j \leq k$ is $\varepsilon$-regular.

\begin{lem}[Duke--Lefmann--R\"odl \cite{DuLeRo}]\label{dlr}
	For any $0<\zeta<\frac 12$ and any $k \in \N$, let $M=\zeta^{-k^2 \zeta^{-5}}$. Suppose $U_1,\ldots,U_k$ are disjoint vertex subsets of a graph $G$. Then there is a partition $\P$ of the cylinder $U_1 \times \dotsb \times U_k$ into at most $M$ parts, each a cylinder of the form $W_1 \times \dotsb \times W_k$ with $W_i \subseteq U_i$, such that the following hold:
	\begin{enumerate}
		\item All but a $\zeta$-fraction of the tuples $(v_1,\ldots,v_k) \in U_1 \times\dotsb \times U_k$ are contained in $\zeta$-regular parts of $\P$ and
		\item For each $W_1 \times \dotsb \times W_k \in \P$ and each $i \in [k]$, $|W_i|\geq |U_i|/M$. 
	\end{enumerate}
\end{lem}

Using this lemma, we now show that we can find a great configuration in $R$, provided that $R$ contains a reasonable number of blue $K_k$. 

\begin{lem}\label{goodconfig}
	Let $0<\zeta<\frac 12$, $0<\delta<2^{-2k^2}$, and $\alpha \geq 2\delta+\zeta k^2$.
	Suppose the set $R$ from (\ref{eq:def-of-R}) contains at least $\alpha |R|^k$ blue $K_k$ and that $\eta\leq \min\{\alpha,\zeta^{2k^2 \zeta^{-5}},\frac 14\}$.
	Then there exist disjoint $C_1,\ldots,C_k \subseteq R$ which form a $(k,\zeta,\delta)$-great configuration.
\end{lem}

\begin{proof}
	First, we show that since $R$ contains ``many'' blue $K_k$, we can find distinct $i_1,\ldots,i_k$ so that the red blocks $V_{i_1},\ldots,V_{i_k} \subseteq R$ span ``many'' blue $K_k$, where we say that the tuple \emph{spans} a blue $K_k$ if each part contains one vertex of the $K_k$. Note first that the number of blue $K_k$ with at least two vertices in a fixed part $V_i$ is at most
	\[
		\binom{|V_i|}{2} |R|^{k-2} \leq \frac{|V_i|^2}{2}|R|^{k-2} \leq\frac{(|R|/(m/2))^2}{2}|R|^{k-2}\leq  \frac{2|R|^k}{m^2}
	\]
	and, therefore, the number of blue $K_k$ with at least two vertices in the same part is at most $2|R|^k/m$. Recall from Lemma \ref{lem:partition} that $m=K(\eta)\geq 2^{1/\eta^{(10/\eta)^{12}}}>4/\alpha$, so that $2|R|^k/m < \alpha |R|^k/2$. Thus, at least $\alpha |R|^k/2$ blue $K_k$ go between parts. By averaging over all choices of $i_1,\ldots,i_k$, we find that there must exist a choice such that $(V_{i_1},\ldots,V_{i_k})$ spans at least $\frac \alpha 2 |V_{i_1}| \dotsb |V_{i_k}|$ blue $K_k$. We reorder the parts so that these are $V_1,\ldots,V_k$.

	We now apply Lemma \ref{dlr} with the parameter $\zeta$ as in the statement of the lemma. We thus get a partition $\P$ of $V_1 \times \dotsb \times V_k$, with each part $P_\ell \in \P$ a cylinder $W_{1 \ell} \times \dotsb \times W_{k \ell}$. We can write
	\[
		\#\{\text{blue }K_k\text{ in }V_1\times \dotsb \times V_k\}=\sum_\ell \#\{\text{blue }K_k\text{ in }P_\ell\},
	\]
	where a blue $K_k$ in $P_\ell$ is a blue $K_k$ with one vertex in each $W_{i \ell}$. At most a $\zeta$-fraction of the $k$-tuples in $V_1 \times \dotsb \times V_k$ are contained in irregular parts $P_\ell$. Therefore,
	\[
		\sum_{P_\ell\text{ is } \zeta\text{-regular}} \#\{\text{blue }K_k\text{ in }P_\ell\} \geq \left( \frac \alpha 2- \zeta \right) |V_1| \dotsb|V_k| .
	\]
	For each $\zeta$-regular $P_\ell$, we can count the number of blue $K_k$ using Lemma \ref{countinglemma}. This implies that
	\[
		\#\{\text{blue }K_k\text{ in }P_\ell\} \leq \left(\zeta \binom k2+\prod_{1 \leq i<j \leq k} d_B(W_{i\ell},W_{j \ell})\right) \prod_{i=1}^k |W_{i \ell}| ,
	\]
	where $d_B$ denotes the blue density. Therefore,
	\begin{align*}
		\left( \frac \alpha 2- \zeta \right) |V_1| \dotsb|V_k|  &\leq \sum_{P_\ell\, \zeta\text{-regular}}\left(\zeta \binom k2+ \prod_{i<j} d_B (W_{i \ell},W_{j \ell})\right) \prod_{i=1}^k |W_{i \ell}|\\
		&\leq \max_{P_\ell\, \zeta\text{-regular}} \left(\zeta \binom k2+ \prod_{i<j} d_B(W_{i \ell},W_{j \ell}) \right) \sum_{P_\ell \in \P} \prod_{i=1}^k |W_{i \ell}|\\
		&=\max_{P_\ell\, \zeta\text{-regular}} \left( \zeta \binom k2+\prod_{i<j} d_B(W_{i \ell},W_{j \ell}) \right) |V_1|\dotsb|V_k|.
	\end{align*}
	Thus, there is some $\ell$ for which $P_\ell$ is $\zeta$-regular and
	\[
		\prod_{i<j} d_B(W_{i \ell},W_{j \ell}) \geq \left(\frac \alpha 2- \zeta\right)- \zeta \binom k2 \geq \frac\alpha2- \zeta \frac{k^2}2 \geq \delta, 
	\]
	by our choice of $\alpha$. Setting $C_i=W_{i \ell}$, we have found sets $C_1,\ldots,C_k$ such that each pair $(C_i,C_j)$ is $\zeta$-regular and the product of their pairwise blue densities is at least $\delta$.

	We also know that for each $i$, $|C_i| \geq |V_i|/M$. Since $V_i$ was $\eta$-regular, the hereditary property of regularity implies that $C_i$ is $\eta M$-regular and, by our choice of $\eta\leq\zeta^{2k^2 \zeta^{-5}}=M^{-2}$, it will thus be $\zeta$-regular. Moreover, since $|C_i| \geq \eta|V_i|$, the $\eta$-regularity of $V_i$ implies that the red density of $C_i$ is at least $\frac 12 -\eta \geq \delta^{1/\binom k2}$. 
	These are the properties defining great configurations, so we see that $C_1,\ldots,C_k$ is a $(k,\zeta,\delta)$-great configuration, as desired.
\end{proof}

Thus, by assuming that $R$ contains many blue $K_k$, we conclude that it also contains a great configuration. By Lemma \ref{lem:goodconfigsuffices} and our choice of $\delta=2^{-2k-3}\varepsilon$ and $\zeta=\delta^3/k^2$, this $(k,\zeta,\delta)$-great configuration then implies the existence of the required monochromatic copy of $B_n \up k$. In the next subsection, we will see how to find such a book under the opposite assumption that $R$ has fewer than $\alpha |R|^k$ blue $K_k$.

\subsection{Few blue cliques}\label{fewbluekks}

We now assume that the condition of Lemma \ref{goodconfig} is not met and show that we can still find a monochromatic book, though we can no longer guarantee the existence of a great configuration (for instance, if every edge is red).
Broadly speaking, the idea of the proof is to use the assumption that $R$ contains few blue $K_k$ to find either a monochromatic book or a large subset of $R$ with few blue $K_{k-1}$. Applying the same argument repeatedly (starting from a set with few $K_r$ and restricting to a large subset with few $K_{r-1}$), we will eventually find a large subset of $R$ with few blue $K_2$, that is, few blue edges. At that point, it is straightforward to show that this set must contain a large red book, which concludes the proof. 

A pair of vertex sets $(X,Y)$ is said to be \emph{lower-$(\lambda,\gamma)$-regular} if $d(X',Y') \geq \lambda$ holds for every $X' \subseteq X$, $Y' \subseteq Y$ with $|X'| \geq \gamma |X|$, $|Y'| \geq \gamma |Y|$. We begin by showing that given any two vertex sets, one of which is regular and fairly dense in red, we can find either a large red book or a large pair of lower-$(\lambda,\gamma)$-regular subsets in the blue graph.

\begin{lem}\label{lowerreg}
	Fix $0<\beta\leq \frac 12$, $0<\gamma <\frac 15$, and $0<\lambda<\frac 1{12k}$, and set 
	\[
		\beta'=\frac{\beta}{1-2k \lambda} \qquad \text{ and } \qquad \rho=\left( \frac \gamma 2 \right) ^{1-\frac{\ln (1-\beta')}{\gamma}}.
	\] 
	Fix $0<\eta<\rho/2^{2k^2}$. Suppose $A$ is a set that is $\eta$-regular with red density at least $\frac 13$ and $B$ is a disjoint set of vertices with $|B|\geq \gamma^{-3}$. Then either there is a red $K_k$ in $A$ with at least $\beta|B|$ red extensions in $B$ (i.e.,\ a red book $B_{\beta|B|} \up k$) or there are subsets $A' \subseteq A$, $B' \subseteq B$ such that $|A'| \geq \rho|A|$, $|B'| \geq (1-\gamma) (1-\beta')^{1+\gamma}|B|$, and $(A',B')$ is lower-$(\lambda,\gamma)$-regular in blue.
\end{lem}

\begin{proof}
	We will iteratively build two sequences of vertex sets $A=A_0 \supseteq A_1\supseteq \dotsb$ and $B=B_0\supseteq B_1\supseteq \dotsb$ with the following properties:
	\begin{enumerate}
	\renewcommand\theenumi{\roman{enumi}}
	\renewcommand\labelenumi{(\theenumi)}
		\item \label{property:asize} $|A_\ell| \geq (\gamma/2)^\ell |A|$,
		\item \label{property:bsize} $(1-\gamma)^{\ell}|B|-\ell \leq |B_\ell| \leq (1- \gamma)^\ell |B|$, and
		\item \label{property:degree} Setting $\ol{B_\ell}=B \setminus B_\ell$, every vertex in $A_\ell$ has blue degree at most $2 \lambda |\ol{B_\ell}|$ into $\ol{B_\ell}$.
	\end{enumerate}
	In each step of the process, either $(A_\ell,B_\ell)$ will be lower-$(\lambda,\gamma)$-regular in blue (in which case we take $A'=A_\ell$, $B'=B_\ell$) or else we will be able to continue the sequence. If we continue for sufficiently long, then the outcome will yield the desired large red book.

	To begin, set $A_0=A, B_0=B$, noting that the three properties we are tracking hold vacuously, since $|A_0|=|A|$, $|B_0|=|B|$, and $\ol{B_0}=\varnothing$. Suppose now that we have defined $A_\ell$ and $B_\ell$ satisfying properties (\ref{property:asize})--(\ref{property:degree}). If $(A_\ell,B_\ell)$ is lower-$(\lambda,\gamma)$-regular in blue, then we output $(A_\ell,B_\ell)$ as our desired pair $(A',B')$. If not, we may find $X \subseteq A_\ell$, $Y \subseteq B_\ell$ such that $|X| \geq \gamma |A_\ell|$, $|Y| \geq \gamma |B_\ell|$, and $d_B(X,Y)<\lambda$. If $Z$ is a uniformly random subset of $Y$ of cardinality exactly $\lceil \gamma |B_\ell|\rceil$, then $\E[d_B(X,Z)]=d_B(X,Y)$, which implies that there is some subset $Y' \subseteq Y$ with $|Y'|=\lceil \gamma |B_\ell|\rceil$ and $d_B(X,Y') \leq d_B(X,Y)<\lambda$. Fix such a $Y'$. 

	Let $X_1 \subseteq X$ be the set of all $x \in X$ with $e_B(x,Y')<2 \lambda |Y'|$ and $X_2=X \setminus X_1$. Then 
	\begin{align*}
		\lambda |X||Y'| >d_B(X,Y') |X||Y'|&=\sum_{x \in X} e_B(x,Y')\geq \sum_{x \in X_2} e_B(x,Y') \geq 2 \lambda |X_2||Y'|,
	\end{align*}
	which implies that $|X_2| < \frac 12 |X|$ and thus that $|X_1| \geq \frac 12 |X|$. We set $A_{\ell+1}=X_1$ and $B_{\ell+1}=B_\ell \setminus Y'$.

	We need to check that properties (\ref{property:asize})--(\ref{property:degree}) still hold for $(A_{\ell+1},B_{\ell+1})$. Property (\ref{property:asize}) is rather straightforward, since
	\[
		|A_{\ell+1}|=|X_1| \geq \frac 12 |X| \geq \frac \gamma 2 |A_{\ell}| \geq \frac \gamma 2 \left( \frac \gamma 2 \right) ^\ell |A|=\left( \frac \gamma 2 \right) ^{\ell+1}|A|,
	\]
	where we used our assumption that property (\ref{property:asize}) holds for $A_\ell$. Similarly, 
	\begin{align*}
		|B_{\ell+1}|&=|B_\ell|-|Y'| = |B_\ell|-\lceil \gamma |B_\ell| \rceil \leq |B_\ell|-\gamma |B_\ell|=(1-\gamma)|B_\ell| \leq (1- \gamma)^{\ell+1} |B|
	\end{align*}
	and
	\begin{align*}
		|B_{\ell+1}|& \geq |B_\ell| - \gamma |B_\ell|-1\geq (1-\gamma)\left( (1-\gamma)^\ell|B|-\ell \right) -1\geq (1-\gamma)^{\ell+1}|B|-(\ell+1),
	\end{align*}
	by applying property (\ref{property:bsize}) for $B_\ell$. Finally, if we let $\ol{B_{\ell+1}}=B \setminus B_{\ell+1}$, then we see that
	\[
		\ol{B_{\ell+1}}=B \setminus (B_\ell \setminus Y')=\ol{B_\ell} \sqcup Y'.
	\]
	By applying property (\ref{property:degree}) to $(A_\ell,B_\ell)$, we know that every vertex in $A_\ell$ has blue density at most $2 \lambda$ into $\ol{B_\ell}$. Since $A_{\ell+1} \subseteq A_\ell$, the same holds immediately for all vertices in $A_{\ell+1}$. Additionally, by our choice of $A_{\ell+1}=X_1$, we know that every vertex in $A_{\ell+1}$ has blue density less than $2 \lambda$ into $Y'$. By adding these two facts, we see that $d_B(x,\ol{B_{\ell+1}})<2 \lambda$ for all $x \in A_{\ell+1}$, proving property (\ref{property:degree}). This proves that we can indeed continue the sequence of pairs $(A_\ell,B_\ell)$.

	Now suppose this process continues until step $\ell^*=\lceil- \frac{\ln (1-\beta')}{\gamma}\rceil$. Then 
	\[
		|A_{\ell^*}| \geq \left( \frac \gamma2 \right) ^{\ell^*} |A| \geq \left( \frac \gamma 2 \right) ^{1-\frac{\ln (1-\beta')}{\gamma}}|A|=\rho |A|.
	\]
	Thus, since $A$ was $\eta$-regular and had red density at least $\frac 13$, we see that $A_{\ell^*}$ has red density at least $\frac13- \eta \geq \frac 14$ and is $(\eta/\rho)$-regular. Therefore, by the counting lemma, Lemma \ref{countinglemma}, we see that $A_{\ell^*}$ contains at least
	\[
	    \left(4^{-\binom k2}-\frac \eta \rho \binom k2\right)|A_{\ell^*}|^k \geq \left(2^{-2\binom k2}-2^{-2k^2}\binom k2\right)|A_{\ell^*}|^k >0
	\]
	red $K_k$, since $\eta<\rho/2^{2k^2}$. Thus, $A_{\ell^*}$ contains at least one red $K_k$.

	Fix a red $K_k$ inside $A_{\ell^*}$. Since every vertex in this clique has blue degree at most $2 \lambda |\ol{B_{\ell^*}}|$ into $\ol{B_{\ell^*}}$, they have at least $(1-2k \lambda)|\ol{B_{\ell^*}}|$ common red neighbors inside $\ol{B_{\ell^*}}$. Moreover, by our choice of $\ell^*$ and  property (\ref{property:bsize}), we see that
	\begin{align*}
		|B_{\ell^*}| \leq (1- \gamma)^{\ell^*} |B| \leq e^{-\gamma \ell^*}|B| \leq (1-\beta') |B|,
	\end{align*}
	which implies that the number of red extensions of our fixed clique is at least
	\[
		(1- 2k \lambda) |\ol{B_{\ell^*}}|=(1-2 k \lambda)(|B|-|B_{\ell^*}|)\geq (1-2k \lambda)\beta'|B|=\beta|B|.
	\]
	This gives us our monochromatic red book $B_{\beta|B|}\up k$. 

	Therefore, we may assume that the process stops at some step $\ell\leq \ell^*-1$. Then, by the definition of the sequence, we know that $(A_\ell,B_\ell)$ is lower-$(\lambda,\gamma)$-regular in blue, so all that needs to be done is to check the lower bounds on $|A_\ell|$ and $|B_\ell|$. But, by properties (\ref{property:asize}) and (\ref{property:bsize}), we see that
	\begin{align*}
		|A_\ell|& \geq \left( \frac \gamma 2 \right) ^\ell|A| \geq \left( \frac \gamma 2 \right) ^{\ell^*}|A| \geq \rho|A|
	\end{align*}
	and
	\begin{align}
		|B_\ell|&\geq (1- \gamma)^\ell |B|-\ell \notag\\
		&\geq (1-\gamma)^{-\ln(1-\beta')/\gamma}|B|-\ell \label{eq:replace-ell}\\
		&\geq  e^{(\gamma+\gamma^2) \ln(1-\beta')/\gamma}|B|-\ell\label{eq:exp-lb}\\
		&=(1-\beta')^{1+\gamma}|B|-\ell,\notag
	\end{align}
	where we used the definition of $\ell^*$ in (\ref{eq:replace-ell}) and the inequality $1-x \geq e^{-x-x^2}$, valid for all $x \in [0,\frac 12]$, in (\ref{eq:exp-lb}). Next, we observe that since $\beta \leq \frac 12$ and $\lambda < \frac1{12k}$, we have that $1-\beta'> \frac 25$, which implies that $-\ln(1-\beta')<1$ and thus that $\ell<\frac 1\gamma$. Additionally, since $\gamma<\frac 15$, we have that $\gamma<(\frac 25)^{1+\gamma}<(1-\beta')^{1+\gamma}$. Therefore, since $|B| \geq \gamma^{-3}$, we have that
	\[
	    \ell < \frac 1 \gamma \leq \gamma^2 |B| <\gamma (1-\beta')^{1+\gamma} |B|,
	\]
	which implies that
	\[
	    |B_\ell| \geq (1-\beta')^{1+\gamma}|B|-\ell \geq (1-\gamma)(1-\beta')^{1+\gamma}|B|.
	\]
	This shows that $(A_\ell,B_\ell)$ satisfies the properties required of $(A',B')$ and concludes the proof.
\end{proof}

We will use this lemma in conjunction with the following result, 
which, though stated in a more general form, will tell us that if we have few blue $K_{r+1}$ in a large set, then it has a large subset containing few blue $K_{r}$.

\begin{lem}\label{extcount}
	Fix $0<\lambda\leq \frac 12$, $r \in \N$, and $0<\gamma<\lambda^r$. Suppose that $G=(A,B, E)$ is a bipartite graph such that the pair $(A,B)$ is lower-$(\lambda,\gamma)$-regular. Suppose also that $H=(B,F)$ is an $r$-uniform hypergraph  with vertex set $B$ and edge set $F$. Define an \emph{extension} to be a pair $(a,f) \in A \times F$ such that $a$ is adjacent in $G$ to every vertex of $f$. Then the number of extensions in $G$ is at least
	\[
		\lambda^r |A|\left(|F|- \frac{r\gamma}{\lambda^{r-1}} |B|^r\right).
	\]
\end{lem}
\begin{proof}
	The proof is by induction on $r$. The base case is when $r=1$, which means that $F$ is a $1$-uniform hypergraph on $B$, i.e.,\ simply a subset
	of $B$. If $|F|<\gamma |B|$, then the bound holds trivially, since $\lambda^r|A|(|F|-r \gamma|B|^r/\lambda^{r-1})$ is negative in this case. So suppose that $|F| \geq \gamma |B|$. Then we may apply lower-$(\lambda,\gamma)$-regularity to the pair $(A,F)$ to conclude that $d(A,F) \geq \lambda$. Since $r=1$, the number of extensions is the same as the number of edges between $A$ and $F$. But 
	\[
		e(A,F)=d(A,F)|A||F| \geq \lambda |A||F| \geq \lambda |A|(|F|-\gamma|B|),
	\]
	as desired.

	For the induction step, suppose the lemma is true for $r-1$. Call a vertex $v \in B$ \emph{good} if its degree to $A$ is at least $\lambda|A|$ and \emph{bad} otherwise. Then there are fewer than $\gamma|B|$ bad vertices, for otherwise they would form a set $B'$ of size at least $\gamma|B|$ with density less than $\lambda$ into $A$. For a good vertex $v \in B$, let $A' \subseteq A$ be its set of neighbors in $A$ and let $H_v$ be its \emph{link} in $H$. That is, $H_v=(B,F_v)$ is the $(r-1)$-uniform hypergraph with vertex set $B$ and hyperedges
	\[
		F_v=\left\{ f \in \binom{B}{r-1}: f \cup \{v\} \in F \right\}.
	\]
	By the hereditary property of lower regularity, we know that the pair $(A',B)$ is lower-$(\lambda,\gamma/\lambda)$-regular. Therefore, we may apply the induction hypothesis to the configuration $(A',B,H_v)$ to conclude that the number of extensions of $H_v$ in $A'$ is at least
	\[
		\lambda^{r-1} |A'|\left(|F_v|- \frac{(r-1)(\gamma/\lambda)}{\lambda^{r-2}}|B|^{r-1}\right)\geq\lambda^r |A|\left(|F_v|- \frac{(r-1)\gamma}{\lambda^{r-1}}|B|^{r-1}\right),
	\]
	since $|A'| \geq \lambda |A|$ by the goodness of $v$. Note that, by the definition of $A'$, every extension of $f \in F_v$ into $A'$ yields an extension of $f \cup \{v\} \in F$ into $A$.	Now, instead of counting extensions of hyperedges, it will be convenient to count extensions of \emph{ordered} hyperedges. In other words, every $f \in F$ will be counted $r!$ times, once for each ordering of its vertices. Then, by summing over the first vertex of the ordered hyperedges, we have that
	\begin{align*}
		\#(\text{ordered}\text{ extensions})&=\sum_{v \in B} \#(\text{ordered extensions of hyperedges starting with }v)\\
		&\geq \sum_{v \text{ good}} \#(\text{ordered extensions of hyperedges starting with }v)\\
		&\geq (r-1)!\lambda^r |A| \sum_{v \text{ good}}\left(  |F_v|-\frac{(r-1)\gamma}{\lambda^{r-1}} |B|^{r-1}\right)\\
		&\geq r!\lambda^r|A|\left(|F|-\frac{(r-1)\gamma}{\lambda^{r-1}}|B|^{r}\right)-(r-1)! \lambda^r|A| \sum_{v\text{ bad}} |F_v|\\
		&\geq r! \lambda^r|A|\left(|F|-\frac{(r-1)\gamma}{\lambda^{r-1}}|B|^r\right)-r! \lambda^r |A| (\gamma |B|^r)\\
		&\geq r!\lambda^r |A| \left(|F|-\frac{r \gamma}{\lambda^{r-1}} |B|^r\right),
	\end{align*}
	where we used the fact that since there are at most $\gamma|B|$ bad vertices, there are at most $\gamma |B|^r$ ordered $r$-tuples that start with a bad vertex. Dividing by $r!$ to count unordered extensions gives the desired result. 
\end{proof}

Using these two lemmas, we can tackle the case where $R$ does not contain many blue $K_k$. The main technical details will appear in the next lemma, but for the moment we give a high-level overview. First, since $R$ contains few blue $K_k$, one of its parts spans few blue $K_k$ with the rest of $R$. Call this block $A$ and set $B=R \setminus A$. Then, by Lemma \ref{lowerreg}, either we can find a large red book between $A$ and $B$ or we can restrict to large subsets $A'$ and $B'$ such that $(A', B')$ is lower regular in blue. In the latter case, Lemma \ref{extcount} implies that either there are many blue $K_k$ between $A'$ and $B'$, a possibility which is ruled out by our assumption that $A$ spans few blue $K_k$ with $B$, or else $B'$ must itself contain few blue $K_{k-1}$. 
We now repeat this argument $k-2$ times. At each step, we assume that we have few blue $K_r$ and either we find a large monochromatic book or else we reduce to a large subset with few blue $K_{r-1}$. If we never find a monochromatic book, then, at the end, we find a large subset of $R$ with few blue $K_2$, i.e.,\ a large subset that is close to monochromatic in red. If the parameters are chosen appropriately, we can then show that this large, very red set contains the requisite red $B_n \up k$. The inductive step for this argument is given by the following lemma. Recall that $R$ is the union of blocks $V_i$, each of which is $\eta$-regular and has red density at least $\frac 12$. 

\begin{lem}\label{lem:kr-induction}
	Let $3 \leq r \leq k$ be an integer and let $\beta=1/(k-1)$, $\tau=(1-\beta)^{k-2} \varepsilon/8^k$, $\lambda= 2^{-4k} /k$, $\beta'=\beta/(1-2k \lambda)$ and 
	\[
		\gamma=\min\left\{\lambda^{k^2},1-(1-\beta)^{1/2k}, \frac{(k-7/4)\ln(1-\beta)}{(k-2)\ln(1-\beta')}-1\right\}. 
	\]
	Suppose that $|R| \geq \gamma^{-4}$ and $S \up r \subseteq R$ is a set  that is the disjoint union $S\up r=\bigcup_i V_i \up r$ of blocks, where $V_i \up r \subseteq V_i$ and either $V_i \up r=\varnothing$ or else $|V_i \up r| \geq \tau |V_i|$. Suppose too that
	\[
		|S\up r| \geq \left( (1-\gamma)^{k-r}(1-\beta')^{(k-r)(1+\gamma)}-2(k-r)\tau \right)  |R|
	\]
	and $S\up r$ spans at most $\alpha_r |S\up r|^r$ blue $K_r$, where $\alpha_r =(k-r+1)\lambda^{kr}$. Then either $S\up r$ contains a monochromatic $B_n \up k$ or there is a subset $S \up {r-1} \subseteq S\up r$ that is the union of $V_i \up {r-1} \subseteq V_i \up r$ with either $V_i \up {r-1}=\varnothing$ or $|V_i \up {r-1}| \geq \tau |V_i|$ such that
	\[
		|S\up {r-1}| \geq \left( (1-\gamma)^{k-(r-1)}(1-\beta')^{(k-(r-1))(1+\gamma)}-2(k-(r-1))\tau \right)|R|
	\]
	and $S\up{r-1}$ spans at most $\alpha_{r-1}|S\up{r-1}|^{r-1}$ blue $K_{r-1}$, where $\alpha_{r-1}=(k-(r-1)+1)\lambda^{k(r-1)}$.
\end{lem}

\begin{rem}
	The definition of $\gamma$ is rather complicated, but one can check that for all $k \geq 3$ the first term in the minimum is the smallest, i.e.,\ $\gamma=\lambda^{k^2}$. However, it will be convenient for the proof to define it as above. 
\end{rem}

\begin{proof}[Proof of Lemma \ref{lem:kr-induction}]
	By assumption, $S\up r$ spans at most $\alpha_r |S\up r|^r$ blue $K_r$. 
	Therefore, by averaging, there must be some $i$ for which there are at most $r \alpha_r |V_i \up r||S\up r|^{r-1}$ blue $K_r$ with a vertex in $V_i \up r$.	Fix such an $i$ and let $A=V_i \up r$ and $B=S \up r \setminus V_i \up r$. 
	
	Since $A=V_i \up r \subseteq V_i$ with $|V_i \up r| \geq \tau |V_i|$, we find that $A$ is $(\eta/\tau)$-regular and, since $\tau > \eta$, it has red density at least $\frac 12 - \eta \geq \frac 13$. We apply Lemma \ref{lowerreg} with parameters $\beta, \gamma,$ and $\lambda$ as above to the pair $(A,B)$, which we may do since the assumption that $|R| \geq \gamma^{-4}$ implies that $|B| \geq \gamma^{-3}$. This tells us that we can either find a red book $B\up k_{\beta |B|}$ or subsets $A' \subseteq A$, $B' \subseteq B$ such that $(A', B')$ is lower-$(\lambda,\gamma)$-regular in blue with $|A'| \geq \rho |A|$ and $|B'| \geq (1-\gamma)(1- \beta')^{1+\gamma}|B|$, where
	\[
		\beta'=\frac{\beta}{1-2k \lambda} \qquad \text{ and } \qquad \rho=\left( \frac \gamma 2 \right) ^{1-\frac{\ln (1-\beta')}{\gamma}}.
	\]
	First, suppose that we have found a red book $B\up k_{\beta |B|}$. We know that 
	\begin{align}
		|B|&=|S\up r|-|V_i \up r| \notag\\
		&\geq (1-\gamma)^{k-r}(1- \beta')^{(k-r)(1+\gamma)} |R|-(2(k-r)\tau|R|+|V_i|) \notag\\
		&\geq (1-\gamma)^k (1-\beta')^{(1+\gamma)(k-3)}|R|-2k\tau N\label{eq:plug-in-r}\\
		&\geq (1-\beta)^{1/2}(1-\beta)^{k-5/2}|R|-2k\tau N\label{eq:plug-in-gamma}\\
		&=(1-\beta)^{k-2}|R|-2k\tau N \notag.
	\end{align}
	In (\ref{eq:plug-in-r}), we used that $r \geq 3$ and $|V_i|=N/K(\eta) \leq \tau N$ by the choice of $\tau$. In (\ref{eq:plug-in-gamma}), we used the definition of $\gamma$, which implies that $(1-\gamma)^k \geq (1-\beta)^{1/2}$ and that $(1-\beta')^{(1+\gamma)(k-3)}\geq (1-\beta)^{k-5/2}$, since $1+\gamma \leq \frac{(k-7/4)\ln(1-\beta)}{(k-2)\ln(1-\beta')} \leq \frac{(k-5/2)\ln(1-\beta)}{(k-3)\ln(1-\beta')}$.
	Now, we plug in $|R| \geq N/2$ and our definition of $\tau$ to find that
	\begin{align*}
		|B|&\geq (1-\beta)^{k-2} \frac N2 - 2k(1-\beta)^{k-2} \frac{\varepsilon}{8^k} N\\
		&= (1-\beta)^{k-2} n \left( \left(2^{k-1}+\frac \varepsilon2\right)-2k\varepsilon \frac{2^k+\varepsilon}{8^k} \right) \\
		&\geq (1-\beta)^{k-2} 2^{k-1} n.
	\end{align*}
	Therefore, the number of pages in our red book $B\up k_{\beta |B|}$ satisfies
	\[
		\beta|B| \geq \beta (1- \beta)^{k-2} 2^{k-1}n \geq n,
	\]
	since, for all $k \geq 3$ and $\beta=1/(k-1)$, we have that
	\begin{equation}
		\frac 1{k-1} \left( 1- \frac 1{k-1} \right) ^{k-2} \geq \frac{1}{2^{k-1}}.\label{eq:beta-maximizer}
	\end{equation}
	Thus, in this case, we have found a red $B_n \up k$.

	Therefore, we may suppose that we instead find the subsets $A'$ and $B'$ described earlier. If $B'$ spans fewer than $\frac 12 \alpha_{r-1} |B'|^{r-1}$ blue $K_{r-1}$, then we are done. To see this, we delete from $B'$ all the vertices in blocks $V_i$ that maintain at most a $\tau$ fraction of their vertices and set $S\up{r-1}$ to be the remainder. Doing so discards at most $\tau|R|$ vertices from $B'$, so we find that
	\begin{align*}
	    |S\up{r-1}|&\geq |B'|-\tau |R| \\
	    &\geq (1-\gamma)(1-\beta')^{1+\gamma}\left(|S\up r|-|V_i\up r|\right)-\tau |R|\\
	    &\geq (1-\gamma)(1-\beta')^{1+\gamma}|S\up r|-2\tau |R|\\
	    &\geq (1-\gamma)(1-\beta')^{1+\gamma}\left((1-\gamma)^{k-r}(1-\beta')^{(k-r)(1+\gamma)}-2(k-r)\tau\right)|R|-2\tau |R|\\
	    &\geq \left((1-\gamma)^{k-(r-1)}(1-\beta')^{(k-(r-1))(1+\gamma)}-2(k-(r-1))\tau\right)|R|,
	\end{align*}
	as desired. Additionally, since we discarded at most $\tau |R|$ vertices from $B'$, we discarded at most a $(\tau |R|/|B'|)$-fraction of the vertices in $B'$. By the same computation as in equations (\ref{eq:plug-in-r}) and (\ref{eq:plug-in-gamma}), we see that $|B'| \geq (1-\beta)^{k-2}|R|-2k \tau N \geq \frac 12(1-\beta)^{k-2}|R|$.
	Therefore, 
	$\tau|R|/|B'| \leq 2\tau(1-\beta)^{2-k}=2\varepsilon/8^k<2^{-k}$, so discarding this small fraction of vertices means that $S\up{r-1}$ will still span fewer than $\alpha_{r-1}|S\up{r-1}|^{r-1}$ blue $K_{r-1}$, as needed. Therefore, in this case, we are done.

	So we may assume that $B'$ spans at least $\frac 12 \alpha_{r-1}|B'|^{r-1}$ blue $K_r$. Now we apply Lemma \ref{extcount} to the blue graph between the pair $(A',B')$, which is lower-$(\lambda,\gamma)$-regular by construction. The hypergraph on $B'$ that we will use is the hypergraph $F$ of all blue $K_{r-1}$ in $B'$, i.e.,\ an $(r-1)$-tuple in $B'$ will be a hyperedge if and only if it spans a blue clique. Note that $|F| \geq \frac 12 \alpha_{r-1}|B'|^{r-1}$, since we assumed that $B'$ spans at least that many blue $K_{r-1}$. Then the number of extensions is precisely the number of blue $K_r$ with one vertex in $A'$ and the rest in $B'$. Lemma \ref{extcount} says that this number of extensions is at least
	\begin{align*}
		\lambda^{r-1} |A'| \left( |F|-\frac{(r-1)\gamma}{\lambda^{r-2}}|B'|^{r-1} \right) &\geq \lambda^{r-1} |A'| \left( \frac 12 \alpha_{r-1}-\frac{(r-1)\gamma}{\lambda^{r-2}} \right) |B'|^{r-1} \\
		&\geq \left(\lambda (1-\gamma)(1- \beta')^{1+\gamma}\right)^{r-1} \left(\frac 12 \alpha_{r-1}-\frac{(r-1)\gamma}{\lambda^{r-2}} \right) |A'| |B|^{r-1}\\
		&=:\mu |A'||B|^{r-1}.
	\end{align*}
	In other words, an average vertex in $A'$ is contained in at least $\mu |B|^{r-1}$ blue extensions of a blue $K_{r-1}$ in $B$. Now we delete $A'$ from $A$ and apply this argument again. Formally, set $A_1=A \setminus A'$. Then we apply Lemma \ref{lowerreg} to the pair $(A_1,B)$, which tells us that we either find a red book $B\up k_{\beta|B|}$ or subsets $A_1' \subseteq A_1$, $B' \subseteq B$ such that $(A_1', B')$ is lower-$(\lambda,\gamma)$-regular in blue with $|A'_1| \geq \rho|A_1|$ and $|B'| \geq (1-\gamma) (1- \beta')^{1+\gamma}|B|$. As above, if we find the monochromatic $B\up k_{\beta |B|}$, we are done, since $\beta|B| \geq n$. If not, then either this $B'$ has a density of blue $K_{r-1}$ smaller than $\frac 12 \alpha_{r-1}$, in which case we are again done, or else an average vertex in $A_1'$ is contained in at least $\mu |B|^{r-1}$ blue extensions of a blue $K_{r-1}$ in $B$. In that case, we set $A_2=A_1 \setminus A_1'$ and repeat the process once more. Each time we repeat, either we get the desired conclusion or we can pull out a new subset $A_i' \subseteq A$ with $|A_i'| \geq \rho |A \setminus (A' \cup A_1' \cup \dotsb \cup A_{i-1}')|$ and such that the average vertex in $A_i'$ is contained in at least $\mu |B|^{r-1}$ blue extensions of a blue $K_{r-1}$ in $B$. Since we pull out at least a $\rho$-fraction of the remainder of $A$ at each step, we will eventually pull out at least half the vertices in $A$.
	
	The set $\wt A \subseteq A$ of removed vertices has the property that the average vertex in $\wt A$ is contained in at least $\mu |B|^{r-1}$ blue extensions of a blue $K_{r-1}$ in $B$. Therefore, the total number of extensions between $A$ and $B$ is at least $\frac 12 \mu |A||B|^{r-1}$. However, by construction, this number is also at most $r \alpha_r |A| |S\up r|^{r-1} \leq 2^{r-1} r \alpha_r |A| |B|^{r-1}$, using the fact that $|B| \geq \frac 12 |S\up r|$. So we conclude that 
	\begin{align*}
		2^{r-1}r \alpha_r &\geq \frac 12 \mu=\frac 12\left(\lambda (1-\gamma)(1- \beta')^{1+\gamma}\right)^{r-1} \left(\frac 12 \alpha_{r-1}-\frac{(r-1)\gamma}{\lambda^{r-2}} \right).
	\end{align*}
	Rearranging, this implies that 
	\begin{align}
		\alpha_{r-1} &\leq \frac{2^{r+1} r \alpha_r}{(\lambda(1-\gamma)(1- \beta')^{1+\gamma})^{r-1}}+\frac{2(r-1)\gamma}{\lambda^{r-2}} \notag\\
		&\leq \frac{k2^{4k}}{\lambda^{r-1}}\alpha_r+\frac{2k\gamma}{\lambda^{r-2}}\label{eq:plug-in-rbeta}\\
		&\leq \lambda^{-r} \alpha_r+\lambda^{k^2-r} \label{eq:plug-in-lambda}\\
		&\leq \lambda^{-k} ((k-r+1) \lambda^{kr})+\lambda^{kr-k} \label{eq:inductive-alpha}\\
		&=(k-(r-1)+1) \lambda^{k(r-1)} \notag,
	\end{align}
	where in (\ref{eq:plug-in-rbeta}) we used that $r \leq k$ and $(1-\gamma)(1-\beta')^{1+\gamma} \geq \frac 18$, in (\ref{eq:plug-in-lambda}) we used that $1/\lambda\geq k2^{4k}$ and  $\gamma\leq\lambda^{k^2}<\lambda^{k^2-2}/2k$, and in (\ref{eq:inductive-alpha}) we used our assumption that $\alpha_r \leq (k-r+1)\lambda^{kr}$ and that $k^2-r \geq kr-k$ since $r \leq k$. This is our desired bound. 
\end{proof}

We can now put all the pieces together and finish the proof of Theorem \ref{thm:tripleexp}.

\begin{proof}[Proof of Theorem \ref{thm:tripleexp}]
	Recall that we are given a $2$-coloring of $K_N$, where $N=(2^k+\varepsilon)n$. We set $\delta=2^{-2k-3} \varepsilon$, $\zeta=\delta^3/k^2$, $\eta=\zeta^{2k^2 \zeta^{-5}}$, $\beta=1/(k-1)$, $\tau=(1-\beta)^{k-2}\varepsilon/8^k$, $\lambda= 2^{-4k}/k$,
	$\beta'=\beta/(1-2k \lambda)$, $\alpha=\lambda^{k^2}$ and 
	\[
		\gamma=\min\left\{\lambda^{k^2},1-(1-\beta)^{1/2k},\frac{(k-7/4)\ln(1-\beta)}{(k-2)\ln(1-\beta')}-1\right\}. 
	\]
	We apply Lemma \ref{lem:partition} to our coloring, with $\eta$ as the parameter, which we can do as long as $N \geq 2^{1/\eta^{(10/\eta)^{15}}}$. We assume without loss of generality that at least half the parts in the partition have internal red density at least $\frac 12$ and set $R$ to be the union of these parts. If $R$ spans at least $\alpha|R|^k$ blue $K_k$, then we apply Lemma \ref{goodconfig}. To do so, we need to check that $\alpha \geq 2 \delta+\zeta k^2$, which indeed holds since
	\[
		2 \delta+\zeta k^2=2^{-2k-2}\varepsilon+ 2^{-6k-9}\varepsilon^3 \leq \varepsilon \leq \frac{2^{-4k^3}}{k^{k^2}}=\lambda^{k^2}=\alpha,
	\]
	by our choice of $\varepsilon \leq 2^{-4k^3}/k^{k^2}$ (recall that we are free to make such a choice since we are ultimately interested in small $\varepsilon$). To apply Lemma \ref{goodconfig}, we also need to check that $\eta \leq \min\{\alpha,\zeta^{2k^2 \zeta^{-5}},\frac 14\}$, which certainly holds since $\eta \leq \delta \leq \alpha \leq \frac 14$ by the above. Similarly, $\delta<2^{-2k^2}$ holds since $\delta<\varepsilon<2^{-4k^3}$.
	Thus, Lemma \ref{goodconfig} applies and we may find a $(k,\zeta,\delta)$-great configuration within $R$. We then apply Lemma \ref{lem:goodconfigsuffices} to this $(k,\zeta,\delta)$-great configuration to find the desired monochromatic $B_n \up k$.

	Therefore, we may suppose that $R$ contains fewer than $\alpha |R|^k$ blue $K_k$. We set $S\up k=R$, $\alpha_k=\alpha$, and apply Lemma \ref{lem:kr-induction}. To do so, we need to assume that $|R| \geq \gamma^{-4}$. We then repeatedly apply Lemma \ref{lem:kr-induction} and, at each step of this induction, either we will find a monochromatic $B_n \up k$ or we will be able to continue on to the next step. This process ends when $r=2$, at which point we have found a set $S\up 2$ with fewer than $\alpha_2 |S \up 2|^2$ blue edges, where $\alpha_2\leq k \lambda^{2k} \leq k \lambda^4$, and with
	\begin{align*}
	    |S\up 2| &\geq \left((1-\gamma)^{k-2}(1-\beta')^{(k-2)(1+\gamma)}-2(k-2)\tau\right)|R|\\
	    &\geq \left((1-\beta)^{1/2} (1-\beta)^{k-7/4} - 2k \tau\right)|R|\\
	    &=\left((1-\beta)^{k-5/4} -2k\tau\right)|R|,
	\end{align*}
	using our definition of $\gamma$ as in (\ref{eq:plug-in-gamma}).
	At this point, it is very easy to find a red $B_n \up k$, since $S \up 2$ is almost a red clique. Concretely, first observe that each blue edge is in $\binom{|S\up 2|-2}{k-1}$ tuples of $k+1$ vertices of $S \up 2$. This means that the number of red $K_{k+1}$ in $S\up 2$ is at least
	\[
		\binom{|S\up 2|}{k+1}- \alpha_2 |S\up 2|^2 \binom{|S\up 2|-2}{k-1} \geq \left( 1-2k(k+1)\alpha_2 \right) \binom{|S\up 2|}{k+1}.
	\]
	This implies that a random $k$-tuple of vertices in $S\up k$ is in at least $(1-3k(k+1)\alpha_2) |S\up 2|$ red $K_{k+1}$, on average. Moreover,
	\begin{align*}
		(1-3k(k+1)\alpha_2)|S \up 2| &\geq (1-3k(k+1)k \lambda^4)((1- \beta)^{k-5/4}-2k\tau) |R|\\
		&\geq (1-2^{-10 k})2^{9/8-k} \left( 2^{k-1}+ \frac \varepsilon 2 \right) n\\
		&\geq n,
	\end{align*}
	using, similarly to (\ref{eq:beta-maximizer}), that, by our choice of $\beta=1/(k-1)$, we have that $(1-\beta)^{k-5/4}-2k\tau \geq 2^{9/8-k}$. Therefore, if the process is allowed to continue until $r=2$, then we again find our desired monochromatic book. 

	In this proof, our two lower-bound assumptions on $N$ were that $N \geq 2^{1/\eta^{(10/\eta)^{15}}}$, in order to apply Lemma \ref{lem:partition},
	and that $|R| \geq \gamma^{-4}$, in order to apply Lemma \ref{lem:kr-induction}. The latter is a much weaker condition, since $|R| \geq N/2$ and
	\[
		\gamma^{-4}=\lambda^{-4k^2}=k^{4k^2} 2^{16k^3} \leq \varepsilon^{-4} \ll  2^{1/\eta^{(10/\eta)^{15}}}.
	\]
	Therefore, the lower bound the proof gives is $N \geq 2^{1/\eta^{(10/\eta)^{15}}}$. By our choice of $\varepsilon\leq 2^{-4k^3}/k^{k^2}$, we have that $\zeta =2^{-6k-9} \varepsilon^3/k^2 \geq \varepsilon^4$ and
	\[
		\eta=\zeta^{2k^2 \zeta^{-5}}\geq (\varepsilon^{4})^{2k^2 \varepsilon^{-20}}=\varepsilon^{8k^2 \varepsilon^{-20}}\geq 2^{-\varepsilon^{-22}}.
	\]
	Therefore,
	\[
		2^{1/\eta^{(10/\eta)^{15}}} \leq 2^{2^{2^{\varepsilon^{-25}}}},
	\]
	so our proof goes through when $n \geq 2^{2^{2^{\varepsilon^{-25}}}}$. That is, if $\varepsilon \leq 2^{-4k^3}/k^{k^2}$ and $n \geq 2^{2^{2^{\varepsilon^{-25}}}}$, then any $2$-coloring of the complete graph on $(2^k+\varepsilon)n$ vertices must contain a monochromatic $B_n \up k$. Therefore, as long as
	\[
	    n \geq 2^{2^{2^{(2^{-4k^3}/k^{k^2})^{-25}}}}=2^{2^{2^{k^{25k^2}2^{100 k^3}}}},
	\]
    this result will hold. Thus, for $n \geq 2^{2^{2^{k^{25k^2}2^{100 k^3}}}}$, we have
	\[
		r(B_n \up k) \leq 2^k n+\frac{n}{(\log \log \log n)^{1/25}},
	\]
	as desired.
\end{proof}

\section{Quasirandomness results}\label{sec:quasirandomness}

\subsection{The main result}

We begin by recalling the definition of quasirandomness from the introduction. Usually, quasirandomness is defined for a sequence of graphs and the right-hand side of the defining inequality is just $o(N^2)$. However, it will be more convenient for us to explicitly track the error parameter $\theta$, rather than hiding it in the little-$o$ notation.

\begin{Def}
	For any $\theta>0$, a $2$-coloring of the edges of $K_N$ is called \emph{$\theta$-quasirandom} if, for any disjoint $X$, $Y \subseteq V(K_N)$,
	\[
		\left|e_B(X,Y)- \frac12|X||Y| \right| \leq \theta N^2,
	\]
	where $e_B(X,Y)$ denotes the number of blue edges between $X$ and $Y$. Since the colors are complementary, this condition is equivalent to the analogous condition for the red edge count $e_R(X,Y)$. 
\end{Def}

\noindent With this, we can restate the main theorem of this section, Theorem \ref{thm:quasirandomness}.

\begin{thm:quasirandomness}
	For any $k \geq 2$ and any $0<\theta<\frac 12$, there is some $c=c(\theta,k)>0$ such that if a $2$-coloring of $K_N$ is not $\theta$-quasirandom for $N$ sufficiently large, then it contains a monochromatic $B_n\up k$ with $n=(2^{-k}+c)N$. 
\end{thm:quasirandomness}

\begin{rem}
	This result was conjectured by Nikiforov, Rousseau, and Schelp  \cite{NiRoSc}, who proved it in the case $k=2$. This case is very special because of Goodman's formula \cite{Goodman} for counting monochromatic triangles in $2$-colorings, no analogue of which exists for counting monochromatic cliques of larger size. As such, the approach we use to prove Theorem \ref{thm:quasirandomness} is substantially different and more complicated than that in \cite{NiRoSc}, though it is interesting to note that our new technique actually fails for $k=2$. As such, for completeness, we also present a proof of the $k=2$ case (modifying and simplifying the proof from \cite{NiRoSc}) in Section \ref{subsec:proof-of-k2-case}.
\end{rem}

As mentioned in the introduction, we will actually prove a strengthening of Theorem \ref{thm:quasirandomness}, restated here.

\begin{thm:strongquasirandom}
	For any $k \geq 2$ and any $0<\theta<\frac 12$, there is some $c_1=c_1(\theta,k)>0$ such that if a $2$-coloring of $K_N$ is not $\theta$-quasirandom for $N$ sufficiently large, then it contains at least $c_1 N^k$ monochromatic $K_k$, each of which has at least $(2^{-k}+c_1)N$ extensions to a monochromatic $K_{k+1}$.
\end{thm:strongquasirandom}

We will prove the contrapositive: if fewer than $c_1 N^k$ monochromatic cliques in a coloring have at least $(2^{-k}+c_1)N$ extensions, then the coloring is $\theta$-quasirandom. 
First, we apply an argument similar to that in Section \ref{sec:simplifiedproof} to find a good configuration within our coloring. As we know from Lemma \ref{lem:goodconfigsuffices}, this good configuration is enough to guarantee the existence of a monochromatic $B_n \up k$. Lemma \ref{lem:goodconfigsuffices} followed from Lemma \ref{xilemma}, which proves a lower bound for a certain function of real variables $x_1,\ldots,x_k$.
Here we need a stability version of Lemma \ref{xilemma}, which says that if our vector $(x_1,\ldots,x_k)$ is bounded away from $(\frac 12, \dots, \frac 12)$, then the function in Lemma \ref{xilemma} is bounded away from its minimum. Using this, one can strengthen Lemma \ref{lem:goodconfigsuffices} so that our good configuration not only contains a $B_n \up k$, but also guarantees many larger books, unless every part of the good configuration is $\varepsilon$-regular with density close to $\frac 12$ to the entire vertex set of the graph. Thus, under the assumption that our coloring contains few monochromatic $B\up k_{(2^{-k}+c_1)N}$, we can pull out a small part of the graph that is $\varepsilon$-regular to the whole vertex set. We then iterate this argument, repeatedly pulling out parts of the coloring that are $\varepsilon$-regular to $V$, until we have almost partitioned the graph into such a collection of parts. The property these parts satisfy is a form of weak regularity, which will be sufficient to prove that the coloring is quasirandom.

We begin with a simple consequence of Markov's inequality, saying that whenever a random clique among some large set of monochromatic cliques has many monochromatic extensions in expectation, then we can find many cliques with many extensions.

\begin{lem}\label{lem:markov-consequence}
	Let $\kappa,\xi \in (0,1)$, let $0<\nu <\xi$, and suppose that $\Q$ is a set of at least $\kappa N^k$ monochromatic $K_k$ in a $2$-coloring of $K_N$. Suppose that a uniformly random $Q \in\Q$ has at least $\xi N$ monochromatic extensions in expectation. Then the coloring contains at least $(\xi-\nu )\kappa N^k$ monochromatic $K_k$, each with at least $\nu N$ monochromatic extensions.
\end{lem}

\begin{proof}
	Let $X$ be the random variable counting the number of monochromatic extensions of a random $Q \in \Q$ and let $Y=N-X$. Then $Y$ is a nonnegative random variable with $\E[Y]=N-\E[X] \leq (1-\xi)N$. By Markov's inequality,
	\[
		\pr(X \leq \nu  N)=\pr \left( Y\geq (1-\nu ) N \right) \leq \frac{\E[Y]}{(1-\nu )N} \leq \frac{(1-\xi)N}{(1-\nu )N}=\frac{1-\xi}{1-\nu }.
	\]
	Thus, 
	\[
		\pr(X \geq \nu  N)\geq 1- \frac{1-\xi}{1-\nu }=\frac{\xi-\nu }{1-\nu } \geq \xi-\nu ,
	\]
	which implies that the number of $Q \in \Q$ with at least $\nu N$ extensions is at least $(\xi-\nu )|\Q| \geq (\xi-\nu )\kappa N^k$, as desired.
\end{proof}

By using a stability version of Jensen's inequality, we can obtain the promised stability variant of Lemma \ref{xilemma}, which is stated below. The proof is given in the appendix. 

\begin{lem}\label{stablexilemma}
	Let $k \geq 3$. Then, for every $\varepsilon_0>0$, there exists $\delta_0>0$ such that, for any $x_1,\ldots,x_k \in [0,1]$ with $|x_j- \frac 12| \geq \varepsilon_0$ for some $j$,
	\[
		\prod_{i=1}^k x_i +\frac 1k \sum_{i=1}^k (1-x_i)^k \geq 2^{1-k}+\delta_0.
	\]
\end{lem}

\begin{rem}
	This lemma is actually false for $k=2$, since 
	the minimum value of $\frac 12$ is attained everywhere on the line $x_1+x_2=1$. Note also that the precise numerical dependence between $\delta_0$ and $\varepsilon_0$ depends in a complicated way on $k$, but it is of the form $\delta_0=\Omega_k(\varepsilon_0^2)$.
\end{rem}

Using Lemma \ref{stablexilemma}, we can prove the strengthening of Lemma \ref{lem:goodconfigsuffices} alluded to earlier, which says that a good configuration whose blue density to the rest of the graph is not close to $\frac 12$ actually yields a monochromatic book with substantially more than $2^{-k}N$ pages.

\begin{lem}\label{densityhalf}
	Let $0<\varepsilon_0<\frac 14$ and let $\delta_0=\delta_0(\varepsilon_0)$ be the parameter from Lemma \ref{stablexilemma}. Suppose $\delta\leq \delta_0 \varepsilon_0/2$,  $\eta\leq \delta^{2k^2}$, and $C_1,\ldots,C_k$ is a $(k,\eta,\delta)$-good configuration in a $2$-coloring of $K_N$. Define
	\[
		B_i=\{v\in K_N: |d_B(v,C_i)-\tfrac 12| \geq \varepsilon_0\}.
	\]
	If $|B_i| \geq \varepsilon_0 N$ for some $i$, then the coloring contains a monochromatic book $B_{(2^{-k}+c)N}\up k$, where $c=\delta_0 \varepsilon_0/4$. Moreover, if $|C_i| \geq \alpha N$ for all $i$ and some $\alpha>0$, then there exists some $0<c_1<c$ depending on $\varepsilon_0$, $\alpha$, and $\delta$ such that the coloring contains at least $c_1 N^k$ monochromatic $K_k$, each of which has at least $(2^{-k}+c_1) N$ monochromatic extensions.
\end{lem}

\begin{proof}
	First, as in the proof of Lemma \ref{lem:goodconfigsuffices}, observe that by the counting lemma, Lemma \ref{countinglemma}, the number of blue $K_k$ with one vertex in each $C_i$ is at least
	\begin{align*}
		\left( \prod_{1\leq i<j \leq k}d_B(C_i,C_j)-\eta \binom k2 \right) \prod_{i=1}^k |C_i|&\geq \left( \delta^{\binom k2}-\eta \binom k2 \right) \prod_{i=1}^k|C_i| >0,
	\end{align*}
	so there is at least one blue $K_k$ spanning $C_1,\ldots,C_k$. Similarly, the number of red $K_k$ inside a given $C_i$ is at least
	\[
		\left( d_R(C_i)^{\binom k2}-\eta \binom k2 \right)|C_i|^k \geq \left( \delta^{\binom k2}-\eta \binom k2 \right) |C_i|^k >0,
	\]
	so each $C_i$ contains at least one red $K_k$.

	Suppose, without loss of generality, that $|B_1| \geq \varepsilon_0 N$. For every $v \in V$, let $x_i(v)=d_B(v,C_i)$. By Lemma \ref{stablexilemma}, for $v \in B_1$, we have
	\[
		\prod_{i=1}^k x_i(v)+\frac 1k \sum_{i=1}^k (1-x_i(v))^k \geq 2^{1-k}+\delta_0,
	\]
	where $\delta_0>0$ depends on $\varepsilon_0$. Additionally, for every $v \in V \setminus B_1$, we know, by Lemma \ref{xilemma}, that
	\[
		\prod_{i=1}^k x_i(v)+ \frac 1k \sum_{i=1}^k (1-x_i(v))^k \geq 2^{1-k}.
	\]
	Summing these two inequalities over all vertices, we get that
	\[
		\sum_{v \in V} \prod_{i=1}^k x_i(v)+\frac 1k \sum_{i=1}^k \sum_{v \in V} (1-x_i(v))^k \geq 2^{1-k}N+\delta_0 |B_1|\geq (2^{1-k}+\delta_0 \varepsilon_0)N.
	\]
	Now, we argue as in the proof of Lemma \ref{lem:goodconfigsuffices}. One of the two summands above must be at least $(2^{-k}+\frac {\delta_0 \varepsilon_0} 2)N$. If it is the first, we apply Corollary \ref{cor:randomclique}, using the fact that $\prod_{i<j} d_B(C_i,C_j) \geq \delta^{\binom k2}$ and $\eta \leq \delta^{2k^2} \leq (\delta^{\binom k2})^3/k^2$. Then summing the result of Corollary \ref{cor:randomclique} over all $v \in V$ implies that a random blue $K_k$ spanning $C_1,\ldots,C_k$ will have in expectation at least
	\begin{align*}
		\sum_{v \in V} \left( \prod_{i=1}^k x_i(v)-4 \delta^{\binom k2} \right) \geq \left(2^{-k}+\frac{\delta_0 \varepsilon_0}2- 4 \delta^{\binom k2}\right)N \geq (2^{-k}+c)N
	\end{align*}
	blue extensions, since $c=\delta_0 \varepsilon_0/4$. On the other hand, if the second summand is the larger one, then we find that a random red $K_k$ inside some $C_i$ will have in expectation at least\footnote{As in the proof of Lemma \ref{lem:goodconfigsuffices}, $d_R(v,C_i)$ is not quite the same as $1-d_B(v,C_i)$ if $v \in C_i$, but this discrepancy can be absorbed into the error term.} 
	\[
		\sum_{v \in V} \left( (1-x_i(v))^k-4 \delta^{\binom k2} \right) \geq \left(2^{-k}+\frac{\delta_0 \varepsilon_0}2- 4 \delta^{\binom k2}\right)N \geq (2^{-k}+c)N
	\]
	red extensions.

	To prove the last statement in the lemma, that we can in fact find $c_1 N^k$ monochromatic cliques each with at least $(2^{-k}+c_1) N$ monochromatic extensions, we apply Lemma \ref{lem:markov-consequence}, using the fact that the argument just presented actually finds a set of cliques with at least $(2^{-k}+c)N$ extensions in expectation. To apply Lemma \ref{lem:markov-consequence}, we need only check that the sets of cliques in question are large, of size at least $\Omega(N^k)$. But this again follows from the counting lemma, Lemma \ref{countinglemma}. Specifically, since $|C_i| \geq \alpha N$ for all $i$, Lemma \ref{countinglemma} implies that the number of blue $K_k$ spanning $C_1,\ldots,C_k$ is at least
	\[
		\left( \prod_{1 \leq i<j\leq k} d_B(C_i,C_j)-\eta \binom k2 \right) \prod_{i=1}^k |C_i| \geq \left( \delta^{\binom k2}-\eta \binom k2 \right) \alpha^k N^k \geq \frac 12 \delta^{\binom{k}{2}} \alpha^k N^k
	\]
	and similarly for the number of red $K_k$ inside $C_i$. 
	Thus, if we define $\kappa = \frac 12 \delta^{\binom{k}{2}} \alpha^k$,
	then we find that all the sets of cliques we considered in the above argument contain at least $\kappa N^k$ cliques. Applying Lemma \ref{lem:markov-consequence} with this choice of $\kappa$, $\xi=2^{-k}+c$, and $\nu=2^{-k}+c/2$, we see that the coloring contains at least $(\xi-\nu)\kappa N^k$ monochromatic cliques, each with at least $\nu N$ monochromatic extensions. Taking $c_1=\min\{(\xi-\nu)\kappa,\nu-2^{-k}\} = c\kappa/2$ gives the result. 
\end{proof}

The previous lemma shows that if our coloring has no monochromatic $B\up k_{(2^{-k}+c)N}$, then, for every good configuration $C_1,\ldots,C_k$, most vertices have blue density into $C_i$ that is close to $\frac 12$. The next lemma strengthens this, saying that, in fact, the good configuration is regular to the rest of the graph. We say that a pair $(X,Y)$ of vertex sets in a graph is \emph{$(p,\varepsilon)$-regular} if, for every $X' \subseteq X$, $Y' \subseteq Y$ with $|X'| \geq \varepsilon |X|$, $|Y'| \geq \varepsilon |Y|$, we have
\[
	|d(X',Y')-p| \leq \varepsilon.
\]
Saying that $(X,Y)$ is $(p,\varepsilon)$-regular is essentially equivalent to saying that $(X,Y)$ is $\varepsilon$-regular and has edge density $p \pm \varepsilon$, 
so it is not strictly necessary to make this new definition. However, the next lemma is most conveniently stated in the language of $(p,\varepsilon)$-regularity, rather than that
of $\varepsilon$-regularity.

\begin{lem}\label{lem:regularityofgoodconfig}
	Fix $0<\varepsilon_1<\frac 14$, let $\varepsilon_0=\varepsilon_1^2/2$, and let $\delta_0=\delta_0(\varepsilon_0)$ be the parameter from Lemma \ref{stablexilemma}. Let $0<\delta\leq \delta_0 \varepsilon_0/4=O_k(\varepsilon_1^6)$ and $0<\eta\leq  2^{-2k^2}\varepsilon_1 \delta^{2k^2}$ be other parameters and suppose that $C_1,\ldots,C_k$ is a $(k,\eta,\delta)$-good configuration in a $2$-coloring of $K_N$. If there exists some $i \in [k]$ such that the pair $(C_i,V)$ is not $(\frac 12,\varepsilon_1)$-regular, then the coloring contains a monochromatic $B\up k_{(2^{-k}+c)N}$ for $c=\delta_0 \varepsilon_0/4$. Moreover, if $|C_i| \geq \alpha N$ for all $i$ and some $\alpha>0$, then the coloring contains at least $c_1 N^k$ monochromatic $K_k$, each with at least $(2^{-k}+c_1)N$ extensions, for some $0<c_1<c$ depending on $\varepsilon_1$, $\delta$, and $\alpha$.
\end{lem}

\begin{proof}
	Without loss of generality, suppose that $(C_1,V)$ is not $(\frac 12, \varepsilon_1)$-regular. Then there exist $C_1' \subseteq C_1$ and $D \subseteq V$ such that
	\[
		\left|d_B(C_1',D)-\frac 12 \right| \geq \varepsilon_1,
	\]
	where $|C_1'| \geq \varepsilon_1 |C_1|$ and $|D| \geq \varepsilon_1 N$. Suppose first that $d_B(C_1',D) \geq \frac 12+\varepsilon_1$. Let $D_1 \subseteq D$ denote the set of vertices $v \in D$ with $d_B(v,C_1')< \frac 12 +\frac{\varepsilon_1}{2}$ and $D_2=D \setminus D_1$. Then
	\[
		\left(\frac 12+\varepsilon_1\right) |C_1'||D| \leq \sum_{v \in D_1} e_B(v,C_1')+\sum_{v \in D_2} e_B(v,C_1') \leq \left( \frac 12+ \frac {\varepsilon_1}2 \right) |C_1'||D|+|C_1'||D_2|,
	\]
	which implies that $|D_2| \geq \frac{\varepsilon_1}{2}|D|$.

	Now consider the collection of sets $C_1',C_2,C_3,\ldots,C_k$. By the hereditary property of regularity, this is a $(k,\eta',\delta')$-good configuration, where $\eta'=\eta/\varepsilon_1$ and $\delta'=\delta- \eta$. Moreover, we have a set $D_2 \subseteq V$ with $|d_B(v,C_1') - \frac 12| \geq \varepsilon_1/2$ for all $v \in D_2$ and
	\[
		|D_2| \geq \frac{\varepsilon_1}{2}|D| \geq \frac{\varepsilon_1^2}{2}N,
	\]
	so we may apply Lemma \ref{densityhalf} with $B_1=D_2$ and $\varepsilon_0=\varepsilon_1^2/2$ to conclude that our coloring contains a monochromatic $B\up k_{(2^{-k}+c)N}$ for some $c$ depending on $\varepsilon_0$. If we assume that $|C_i| \geq \alpha N$ for all $i$, then we also get that $|C_1'| \geq \alpha \varepsilon_1 N$, so the second part of Lemma \ref{densityhalf} implies that our coloring contains at least $c_1 N^k$ monochromatic $K_k$, each of which has at least $(2^{-k}+c_1)N$ monochromatic extensions, for some $c_1<c$ depending on $\varepsilon_1$, $\alpha,$ and $\delta$.

	The above argument worked under the assumption that $d_B(C_1',D) \geq \frac 12 +\varepsilon_1$, so we now need to deal with the case where $d_B(C_1',D) \leq \frac 12- \varepsilon_1$. However, this case is similar: we first find a large subset $D_2 \subseteq D$ with $d_B(v,C_1') \leq \frac 12 - \frac{\varepsilon_1}{2}$ for all $v \in D_2$ and then the remainder of the argument is identical. 
\end{proof}

The next technical lemma we need is the following, which spells out the inductive step of the procedure outlined earlier, wherein we repeatedly pull out subsets of our coloring that are regular to the remainder of the graph. 

\begin{lem}\label{lem:pull-out-regular}
	Fix $0<\varepsilon\leq 1/25 k$ and consider a $2$-coloring of $K_N$ with vertex set $V$, where $N$ is sufficiently large in terms of $\varepsilon$. Suppose that $A_1,\ldots,A_\ell$ are disjoint subsets of $V$ such that $(A_i,V)$ is $(\frac 12, \varepsilon^2)$-regular for all $i$. Let $W= V \setminus (A_1 \cup \dotsb \cup A_\ell)$ and suppose that $|W| \geq \varepsilon N$. Then either there is some non-empty $A_{\ell+1} \subseteq W$ such that $(A_{\ell+1},V)$ is $(\frac 12, \varepsilon^2)$-regular or the coloring contains at least $c_1 N^k$ monochromatic $K_k$, each with at least $(2^{-k}+c_1) N$ monochromatic extensions, where $c_1>0$ depends only on $\varepsilon$ and $k$. 
\end{lem}

\begin{proof}
	The structure of this proof is very similar to the structure of Section \ref{sec:simplifiedproof}, where we proved Theorem \ref{thm:conlon}. Let $\varepsilon_1=\varepsilon^2$, $\varepsilon_0=\varepsilon_1^2/2$, and $\delta_0=\delta_0(\varepsilon_0)=\Omega_k(\varepsilon^8)$ be the parameter from Lemma \ref{stablexilemma}. Next, fix $\delta=\delta_0 \varepsilon_0/4$, $\eta= 2^{-2k^2}\varepsilon^2 \delta^{2k^2}$, $c=2^{-k}k \varepsilon^2$, and $c'=4 \varepsilon$ to be other parameters depending on $\varepsilon$ and $k$. We apply Lemma \ref{reglem} to the subgraph induced on $W$, with parameters $\eta$ and $M_0=1/\eta$, to obtain an equitable partition $W=W_1 \sqcup \dotsb \sqcup W_m$, where $M_0 \leq m \leq M=M(\eta,M_0)$. Without loss of generality, we may assume that the parts $W_1,\ldots,W_{m'}$ have internal red density at least $\frac 12$, where $m' \geq m/2$. We build a reduced graph $G$ with vertex set $w_1,\ldots,w_m$, by making $\{w_{j_1}, w_{j_2}\}$ an edge if $(W_{j_1},W_{j_2})$ is $\eta$-regular and $d_B(W_{j_1},W_{j_2}) \geq \delta$. We also set $G'$ to be the subgraph of $G$ induced by $w_1,\ldots,w_{m'}$. 

	We will now show that $G'$ is quite dense and, in fact, that it has few vertices with low degree. Concretely, suppose first that $w_1$ has fewer than $(1-2^{1-k}-2c'-2 \eta)m'$ neighbors in $G'$. Since $w_1$ has at most $\eta m \leq 2 \eta m'$ non-neighbors coming from irregular pairs, this means that there are at least $(2^{1-k}+2c')m'$ parts $W_j$ with $2 \leq j \leq m'$ such that $(W_1,W_j)$ is $\eta$-regular and $d_R(W_1,W_j) \geq 1-\delta$. Let $J$ be the set of these indices $j$ and set $U=\bigcup_{j \in J} W_j$.
	By the counting lemma, Lemma \ref{countinglemma}, $W_1$ contains at least  $\left(2^{-\binom k2}- \eta  \binom k2\right)|W_1|^k$ red copies of $K_k$ and
	\[
		\left(2^{-\binom k2}-\eta  \binom k2\right)|W_1|^k \geq 2^{-k^2} \left( \frac{|W|}{M} \right) ^k \geq \left(\frac{ \varepsilon N}{2^k M}\right)^k,
	\]
	where we used that $\eta\leq \delta^{2k^2} \leq \delta^{\binom k2}/\binom k2$ and $2^{-\binom k2}-\delta^{\binom k2} >2^{-k^2}$, 
	along with our assumption that $|W| \geq \varepsilon N$. If we set $\kappa=( \varepsilon/2^k M)^k$, then this implies that $W_1$ contains at least $\kappa N^k$ red $K_k$. Moreover, $d_R(W_1)^{\binom k2} \geq 2^{-\binom k2} \geq \delta^k/4$. We pick a random such red $K_k$ and apply Corollary \ref{cor:randomclique} with parameters $\eta$ and $\delta^k/4$, which we may do since $\eta \leq (\delta^k/4)^3/k^2$. Then Corollary \ref{cor:randomclique} implies that the expected number of red extensions of this random clique inside $U$ is at least
	\begin{align*}
		\sum_{u \in U} \left(d_R(u,W_1)^k-4 \left(\frac {\delta^k} 4\right)\right) \geq \left( (1- \delta)^k- \delta^k \right) |U| \geq (1-2k \delta) |U|,
	\end{align*}
	where we first used Jensen's inequality applied to the convex function $x \mapsto x^k$ to lower bound $\sum_u d_R(u,W_1)^k$ by $(1- \delta)^k|U|$ and then used that $(1- \delta)^k \geq 1-k \delta$ and $\delta^k \leq k \delta$. Since we assumed that $J$ was large, and since the partition is equitable, we find that
	\[
		|U| \geq (2^{1-k}+2c')m' |W_j|\geq (2^{-k}+c')|W|.
	\]
	Thus, a random red $K_k$ inside $W_1$ has, on average, at least $(1-2k \delta)(2^{-k}+c')|W|$ red extensions in $W$. 

	Now suppose that instead of just $w_1$ having low degree in $G'$, we have a set of at least $\varepsilon m$ vertices $w_j \in V(G')$, each with fewer than $(1-2^{1-k}-2c'-2 \eta)m'$ neighbors in $G'$. Let $S$ be the set of these $j$ and $T=\bigcup_{j \in S} W_j$. By the above argument, for every $j\in S$, we have that $W_j$ contains at least $\kappa N^k$ red $K_k$, each with at least $(1-2k \delta)(2^{-k}+c')|W|$ red extensions on average. Moreover, we have that 
	\[
		|T|=|S||W_j| \geq (\varepsilon m) \frac{|W|}{m}=\varepsilon |W| \geq \varepsilon^2 |V|.
	\]
	So we may apply the $(\frac12,\varepsilon^2)$ regularity of $(A_i, V)$ to conclude that $d_B(T,A_i)=\frac 12 \pm \varepsilon^2$ for all $i$. Thus, if we pick $j \in S$ randomly, then $\E[d_B(W_j,A_i)]=\frac 12 \pm \varepsilon^2$. Therefore, if we first sample $j \in S$ randomly and then pick a random red $K_k$ inside $W_j$, then Corollary \ref{cor:randomclique} implies that this random red $K_k$ will have in expectation at least 
	\begin{align*}
		\sum_{a \in A_i} \left( d_R(a,W_j)^k-4 \left( \frac {\delta^k}4 \right)  \right) \geq \left( \left( \frac 12- \varepsilon^2 \right) ^k- \delta^k \right) |A_i| \geq 2^{-k}(1- 3k \varepsilon^2)|A_i|
	\end{align*}
	red extensions into $A_i$, again by Jensen's inequality. This implies that this random $K_k$ has in expectation at least $(1 -3k \varepsilon^2) 2^{-k}|A_1 \cup \dotsb \cup A_\ell|$ extensions into $A_1 \cup \dotsb \cup A_\ell$. Adding up the extensions into this set and into $W$, its complement, shows that this random red $K_k$ has in expectation at least $\xi N$ red extensions, where $\xi$ is a weighted average of $(1-3k \varepsilon^2)2^{-k}$ and $(1-2k \delta)(2^{-k}+c')$ with the latter quantity receiving weight at least $\varepsilon$, since $|W| \geq \varepsilon N$. Thus,
	\begin{align*}
		\xi &\geq (1- \varepsilon)(1-3k \varepsilon^2) 2^{-k}+\varepsilon(1-2k \delta)(2^{-k}+c')
		\\
		&\geq (1-3k \varepsilon^2- \varepsilon)2^{-k}+\varepsilon(1-2k \delta)(1+2^k c')2^{-k}.
	\end{align*}
	We claim that $2k \delta < 2^k c'/(2(1+2^k c'))$. Indeed, if $2^k c' \geq 1$, then this follows from $2k \delta < 2k \varepsilon < \frac 14$, whereas if $2^k c'<1$, then this follows from $2k \delta < 2^{k} \varepsilon = 2^k c'/4< 2^k c'/(2(1+2^k c'))$. Therefore,
	\[
	    (1- 2k \delta)(1+2^k c') > \left(1 - \frac{2^k c'}{2(1+2^k c')}\right)(1+2^k c') = 1+2^{k-1}c'.
	\]
    Continuing the above computation, we thus have that
    \begin{align*}
        \xi &\geq (1-3k \varepsilon^2- \varepsilon)2^{-k}+\varepsilon(1-2k \delta)(1+2^k c')2^{-k}\\
        &\geq 2^{-k}(1 - 3k \varepsilon^2 +2^{k-1} c' \varepsilon)\\
        &\geq 2^{-k} (1- 3k \varepsilon^2 + 2^{k+1} \varepsilon^2)\\
        &\geq 2^{-k} + c,
    \end{align*}
	where in the last step we used that $2^{k+1} \geq 4k$ for $k \geq 2$, as well as the definition of $c=2^{-k} k \varepsilon^2$.
	
	Recall that we picked this random red $K_k$ by first uniformly sampling a random $j \in S$ and then picking a uniformly random red $K_k$ in $W_j$. By the above, this random red $K_k$ has at least $(2^{-k}+c)N$ red extensions in expectation. Therefore, there must exist some $j \in S$ such that a random red $K_k$ in $W_j$ has at least $(2^{-k}+c)N$ red extensions in expectation. If we let $\Q$ denote the set of red $K_k$ in $W_j$, then $|\Q| \geq \kappa N^k$ by our earlier computation.
	Therefore, by Lemma \ref{lem:markov-consequence}, we can find at least $c_1 N^k$ red $K_k$s, each with at least $(2^{-k}+c_1)N$ red extensions, for some $c_1<c$ depending 
	only on $\varepsilon$ and $k$. 

	Hence, we may assume that at most $\varepsilon m \leq 2 \varepsilon m'$ vertices of $G'$ have degree less than $(1-2^{1-k}-2c'-2 \eta)m'$. Therefore, the average degree in $G'$ is at least
	\[
		(1-2^{1-k}-2c'-2 \eta)(1-2 \varepsilon) \geq 1-2^{1-k}-2 c'-2 \eta-2 \varepsilon \geq 1-2^{1-k}-12 \varepsilon,
	\]
	by our choice of $c'=4 \varepsilon$ and $\eta< \varepsilon$. Since $\varepsilon < \frac{1}{25k}$, we have that $12 \varepsilon < \frac 1{2k}$. Therefore, $1-2^{1-k}-12 \varepsilon >1-1/(k-1)$. Thus, by Tur\'an's theorem, $G'$ contains a $K_k$. Let $w_{j_1},\ldots,w_{j_k}$ be the vertices of this $K_k$ in $G'$. We claim that $(C_1,\ldots,C_k)=(W_{j_1},\ldots,W_{j_k})$ is a $(k,\eta,\delta)$-good configuration in our original coloring of $K_N$. Indeed, by our definition of the $W_j$, we know that each of them is an $\eta$-regular set with red density at least $\frac 12\geq \delta$ and, by the definition of the reduced graph $G$, we know that if $\{w_j,w_{j'}\}$ is an edge of $G$, then $(W_j,W_{j'})$ is an $\eta$-regular pair with $d_B(W_j,W_{j'}) \geq \delta$. 
	
	Moreover, we know that $|C_j|\geq \alpha N$, where $\alpha=\varepsilon/M$ and $M$ depends only on $\eta$ and, thus, only on $\varepsilon$ and $k$. Therefore, by Lemma \ref{lem:regularityofgoodconfig}, we know that either our coloring contains at least $c_1 N^k$ monochromatic $K_k$, each with at least $(2^{-k}+c_1) N$ monochromatic extensions, or $(C_j,V)$ is $(\frac 12, \varepsilon^2)$-regular for all $j$, where $c_1$ again depends only on $\varepsilon$ and $k$. In particular, we may set $A_{\ell+1}=C_1$ (or any other $C_j$) and obtain the desired conclusion.
\end{proof}

The previous lemma shows that if we assume our coloring does not contain $c_1 N^k$ monochromatic $K_k$, each with at least $(2^{-k}+c_1) N$ monochromatic extensions, then we may inductively pull out subsets that are $(\frac 12, \varepsilon^2)$-regular with the whole vertex set and we may keep doing so until the remainder of the graph becomes too small (namely, until it contains only an $\varepsilon$-fraction of the vertices). At the end of this process, we will have almost partitioned our vertex set into parts which are not necessarily $\varepsilon$-regular with each other, but which satisfy a more global regularity condition (for comparison, this notion is a bit stronger than the Frieze--Kannan notion of weak regularity \cite{FrKa96,FrKa99}). Our final technical lemma shows that this global regularity is enough to conclude that our coloring is $\theta$-quasirandom.

\begin{lem}\label{lem:partition-implies-quasirandomness}
	Let $\varepsilon\leq \theta/2$. Suppose there is a partition
	\[
		V(K_N)=A_1 \sqcup \dotsb \sqcup A_\ell \sqcup A_{\ell+1},
	\]
	where, for each $i\leq \ell$, $(A_i,V)$ is $(\frac 12,\varepsilon)$-regular and $|A_{\ell+1}| \leq \varepsilon N$. Then the coloring is $\theta$-quasirandom.
\end{lem}

\begin{proof}
	Fix disjoint $X,Y \subseteq V(K_N)$. We need to check that
	\[
		\left|e_B(X,Y)- \frac 12 |X||Y| \right| \leq \theta N^2.
	\]
	First, observe that if $|Y| \leq \varepsilon N$, then 
	\[
		\left|e_B(X,Y)- \frac 12 |X||Y| \right| \leq \frac 12 |X| |Y| \leq \frac \varepsilon 2 N^2 \leq \theta N^2.
	\]
	So, from now on, we may assume that $|Y| \geq \varepsilon N$. For $1 \leq i \leq \ell+1$, let $X_i=A_i \cap X$ and define $I_X=\{1 \leq i \leq \ell:|X_i| \geq \varepsilon |A_i|\}$. Then  
	\[
		\sum_{i \notin I_X} |X _i| \leq |A_{\ell+1}|+ \varepsilon \sum_{i=1}^\ell |A_i| \leq 2\varepsilon N.
	\]
	We now write
	\[
		e_B(X,Y)-\frac 12 |X||Y|=\sum_{i=1}^{\ell+1} \left( e_B(X_i,Y)-\frac 12 |X_i||Y| \right) .
	\]
	We split this sum into two parts, depending on whether $i \in I_X$ or not. First, suppose that $i \in I_X$. Then $|X_i| \geq \varepsilon |A_i|$ and $|Y| \geq \varepsilon |V|$ by our assumption that $|Y| \geq \varepsilon N$, so we may apply the $(\frac 12, \varepsilon)$-regularity of $(A_i,V)$ to conclude that
	\[
		\sum_{i \in I_X}\left|e_B(X_i,Y)-\frac 12 |X_i||Y|\right|=\sum_{i \in I_X}\left|d_B(X_i,Y)-\frac 12\right| |X_i||Y| \leq \sum_{i \in I_X}\varepsilon |X_i||Y| \leq \varepsilon |X||Y| \leq \varepsilon N^2.
	\]
	On the other hand, we know by our earlier discussions that
	\[
		\sum_{i \notin I_X} \left| e_B(X_i,Y)- \frac 12 |X_i||Y| \right| \leq \frac 12 |Y|\sum_{i \notin I_X} |X_i| \leq \frac 12 |Y| (2 \varepsilon N) \leq \varepsilon N^2.
	\]
	Adding these up, we conclude that
	\[
		\left| e_B(X,Y) -\frac 12 |X||Y|\right| \leq 2 \varepsilon N^2 \leq \theta N^2,
	\]
	as desired. 
\end{proof}

\begin{rem}
	The output of Lemma \ref{lem:pull-out-regular} is a collection of sets, each of which is $(\frac 12,\varepsilon^2)$-regular to $V$. However, all that Lemma \ref{lem:partition-implies-quasirandomness} requires is $(\frac 12, \varepsilon)$-regularity, which is substantially weaker. The reason for the discrepancy is that Lemma \ref{lem:pull-out-regular} requires a quadratic dependence between the level of regularity and the size of the remainder set $W$.
\end{rem}

With this collection of technical lemmas, we can finally prove the main result of this section.

\begin{proof}[Proof of Theorem \ref{thm:strongquasirandom} for $k \geq 3$]\label{proof-of-thm:strongquasirandom}
	Fix $k \geq 3$ and $0<\theta<\frac 12$, let $\varepsilon=\min\{\theta/2,1/25k\}$, and fix a $2$-coloring of $K_N$. Let $c_1>0$ be the constant from Lemma \ref{lem:pull-out-regular}, depending only on $\varepsilon$ and $k$ and, thus, only on $\theta$ and $k$. If our coloring does not contain at least $c_1 N^k$ monochromatic $K_k$, each with at least $(2^{-k}+c_1)N$ monochromatic extensions, then we may repeatedly apply Lemma \ref{lem:pull-out-regular}. At each step, we find a new set $A_i \subseteq V$ such that $(A_i,V)$ is $(\frac 12, \varepsilon^2)$-regular, as long as the remainder of the graph has cardinality at least $\varepsilon N$. When this is no longer the case, we stop the iteration and apply Lemma \ref{lem:partition-implies-quasirandomness} with $A_1,\ldots,A_\ell$ the sets we pulled out using Lemma \ref{lem:pull-out-regular} and $A_{\ell+1}$ the remainder set of cardinality at most $\varepsilon N$. 
	This implies that the coloring is $\theta$-quasirandom, as desired. 
\end{proof}

\begin{rem}
	The dependence between $c_1$ and $\theta$ in this proof is of tower type, because, in the proof of Lemma \ref{lem:pull-out-regular}, we assumed that $N$ was sufficiently large to apply Lemma \ref{reglem}, which gives a tower-type bound.
	In principle, it should also be possible to obtain a proof
	of Theorem \ref{thm:strongquasirandom} avoiding Lemma \ref{reglem} and only using Lemma \ref{lem:partition}, as in the proof of Theorem \ref{thm:tripleexp} in Section \ref{sec:basicresult}. Doing so would likely give a tighter dependence between $\theta$ and $c_1$ in Theorem \ref{thm:strongquasirandom}, since Lemma \ref{lem:partition} never invokes tower-type dependencies. However, we chose not to pursue this further, because the proof of Theorem \ref{thm:tripleexp} is already substantially more involved than that of Theorem \ref{thm:conlon} (for instance, we have to split into two cases depending on the number of blue $K_k$) and obtaining a stability version would inevitably add further complications.
\end{rem}

\subsection{The \texorpdfstring{$k=2$}{k=2} case}\label{subsec:proof-of-k2-case}

As mentioned previously, our proof of Theorem \ref{thm:strongquasirandom} fails for $k=2$, because Lemma \ref{stablexilemma} is false in that case. However, even though our proof fails, the result is still true, since the following simple variant of the argument in \cite{NiRoSc} applies.

\begin{proof}[Proof of Theorem \ref{thm:strongquasirandom} for $k=2$]
	Consider a red/blue coloring of the edges of $K_N$ in which there are at most $\varepsilon \binom N 2$ edges that are in at least $(\frac{1}{4}+\varepsilon)(N-2)$ monochromatic triangles, where $\varepsilon>1/N$. We may suppose without loss of generality that the red edge density is at least $1/2$. Let $\codeg_R(u,v)$ and $\codeg_B(u,v)$ denote the number of common red and blue neighbors, respectively, of the vertices $u$ and $v$, and let $$S:=\sum_{(u,v)~\textrm{red}} \left(\codeg_R(u,v)-\frac{N-2}{4}\right)_+ + \sum_{(u,v)~\textrm{blue}} \left(\codeg_B(u,v)-\frac{N-2}{4}\right)_+,$$
	where the nonnegative part $x_+$ is given by $x_+=x$ if $x \geq 0$ and $x_+=0$ if $x<0$. 
	The number $S$ is a measure of the monochromatic book excess over the random bound. Since the excess is at most $\varepsilon(N-2)$ for all but $\varepsilon\binom {N}{2}$ edges (whose excess is at most $\frac{3}{4}(N-2)$), we obtain 
	\begin{equation}\label{equation1} S \leq \varepsilon\binom{N} {2}\cdot \frac{3}{4}(N-2)+(1-\varepsilon)\binom{N} {2}\cdot \varepsilon (N-2) < 6\varepsilon \binom{N} {3}.
	\end{equation}
	Let $M$ denote the number of monochromatic triangles in the coloring. Observe that 
	\begin{equation}\label{equation2}
	\frac S3 \geq M-\frac{1}{4}\binom {N} {3},
	\end{equation} 
	as the right-hand side is just one third of the sum that defines $S$ taken without nonnegative parts.  Goodman's formula \cite{Goodman} for the number of monochromatic triangles in a coloring is 
	\[
		M= \frac 12 \left( \sum_{v \in V}\left(\binom{\deg_B(v)}2+\binom{\deg_R(v)}2\right)-\binom N3 \right). 
	\]
	This identity is equivalent to $$M-\frac{1}{4}\binom{N} {3} = -\frac{1}{4}\binom{N} {2}+\frac{1}{2}\sum_{v \in V} \left(\deg_R(v)-\frac{N-1}{2}\right)^2.$$
	From this identity, the inequalities (\ref{equation1}) and  (\ref{equation2}),
	 and the Cauchy--Schwarz inequality, we obtain 
	$$\sum_{v \in V} \left|\deg_R(v)-\frac{N-1}{2}\right|\leq \varepsilon^{1/2} N^2,$$
	where we used the bound $\varepsilon>1/N$.

	By the inclusion-exclusion principle, we have
	 $$\codeg_B(u,v) \geq N-2-\deg_R(u)-\deg_R(v)+\codeg_R(u,v),$$ from which it follows that
	$$\codeg_B(u,v)-\frac{N-2}{4} \geq \codeg_R(u,v)-\frac{N-2}{4} - \left|\deg_R(u)-\frac{N-1}{2}\right|-\left|\deg_R(v)-\frac{N-1}{2}\right|-1.$$
	Summing over all pairs $(u,v)$, we get 
	$$ S \geq \sum_{(u,v)} \left(\codeg_R(u,v)-\frac{N-2}{4}\right)_+ - (N-1)\sum_{u}\left|\deg_R(u)-\frac{N-1}{2}\right|-\binom{N} {2}.$$
	Since the red density is at least $1/2$, convexity implies that the average value of $\codeg_R(u,v)$ is at least $\frac{N-3}{4}$, so
	$$ \sum_{(u,v)}\left(\codeg_R(u,v)-\frac{N-3}{4}\right)_+ \geq \sum_{(u,v)} \frac{1}{2}\left|\codeg_R(u,v)-\frac{N-3}{4}\right|.$$
	Moreover,
	\[
	    \sum_{(u,v)} \left(\codeg_R(u,v)-\frac{N-2}{4}\right)_+-\sum_{(u,v)}\left(\codeg_R(u,v)-\frac{N-3}{4}\right)_+\geq \sum_{(u,v)} -\frac 14\geq -\frac{N^2}8.
	\]
	Putting all this together, we get 
	$$\sum_{(u,v)} \left|\codeg_R(u,v)-\frac{N-3}{4}\right| \leq \frac{N^2}4+ 2S+2(N-1)\varepsilon^{1/2}N^2+N(N-1) \leq \left(4\varepsilon+2\varepsilon^{1/2}\right)N^3.$$
	Chung, Graham, and Wilson \cite{ChGrWi} proved that if an $N$-vertex graph with density at least $1/2$ has average codegree $N/4+o(N)$, then it is $o(1)$-quasirandom. Since the red density is at least $1/2$, we therefore get that the coloring is $\theta$-quasirandom for some $\theta$ depending on $\varepsilon$. 
\end{proof}

\subsection{The converse result}

As mentioned in the introduction, a converse to Theorem \ref{thm:strongquasirandom} is also true, meaning that the condition in Theorem \ref{thm:strongquasirandom} is an equivalent characterization of quasirandomness (up to a change in parameters).

\begin{thm}\label{thm:quasirandomconverse}
	For any $k \geq 2$ and any $c_2 > 0$, there is some $\theta > 0$ such that if a $2$-coloring of the edges of $K_N$ is $\theta$-quasirandom, then the number of monochromatic $K_k$ with at least $(2^{-k}+c_2)N$ monochromatic extensions is at most $c_2 N^k$. 
\end{thm}

\begin{proof}
	We will use the standard result of Chung, Graham, and Wilson that a quasirandom coloring contains roughly the correct count of any fixed monochromatic subgraph. Specifically, for every $\delta>0$, there is some $\theta>0$ such that, in any $\theta$-quasirandom coloring, 
	\begin{align*}
		M(K_k)&=\#(\text{monochromatic }K_k)=2^{1-\binom k2} \binom Nk \pm \delta N^k,\\
		M(K_{k+1})&=\#(\text{monochromatic }K_{k+1})=2^{1-\binom{k+1}2} \binom N{k+1} \pm \delta N^{k+1},\\
		M(K_{k+2}-e)&=\#(\text{monochromatic }K_{k+2}-e)=2^{2-\binom{k+2}2}\binom N{k+2} \binom{k+2}2 \pm \delta N^{k+2},
	\end{align*}
	where $K_{k+2}-e$ is the graph formed by deleting one edge from $K_{k+2}$. Note that for this latter count we have an extra factor of $\binom {k+2}2$, which accounts for the fact that the graph is not vertex transitive. On the other hand, we observe that every monochromatic copy of $K_{k+2}-e$ corresponds to two distinct extensions of a single monochromatic $K_k$ to a monochromatic $K_{k+1}$. Therefore,
	\[
		M(K_{k+2}-e)=\sum_{Q} \binom{\#(\text{monochromatic extensions of }Q)}2,
	\]
	where the sum is over all monochromatic $K_k$. Let $\ext(Q)$ denote the number of monochromatic extensions of $Q$. Then we also have that $\sum_Q \ext(Q)$ counts the total number of ways of extending a monochromatic $K_k$ to a monochromatic $K_{k+1}$, which is precisely $(k+1)M(K_{k+1})$, since each monochromatic $K_{k+1}$ contributes exactly $k+1$ terms to the sum.
	
	Note that we may assume that $N$ is sufficiently large by replacing $c_2$ by a smaller value. Concretely, we assume henceforth that $N\geq \max\{k^2+k, 1/\delta\}$. We now consider the quantity 
	\[
		E=\sum_{Q}
		(\ext(Q)-2^{-k}N)^2.
	\]
	On the one hand, we have that
	\begin{align*}
		E&=\sum_Q \ext(Q)^2 -2^{1-k}N \sum_Q \ext(Q)+\sum_Q 2^{-2k} N^2\\
		&=\left(2 \sum_Q \binom{\ext(Q)}2+\sum_Q \ext(Q) \right)-2^{1-k} N \sum_Q \ext(Q)+2^{-2k} N^2 M(K_k)\\
		&=2 M(K_{k+2}-e)+(1-2^{1-k}N)(k+1)M(K_{k+1})+2^{-2k}N^2 M(K_k)\\
		&\leq 2^{3-\binom{k+2}2} \binom N{k+2} \binom{k+2}2-2^{2-k-\binom{k+1}2}N(k+1) \binom N{k+1}+2^{1-2k-\binom k2} N^2 \binom Nk +2k \delta N^{k+2}\\
		&=2^{2-\binom{k+2}2}\binom Nk (-N +k^2+k)+2k\delta N^{k+2} \\
		&\leq 2k \delta N^{k+2}.
	\end{align*}
	On the other hand, suppose there were at least $c_2 N^k$ monochromatic $K_k$ with at least $(2^{-k}+c_2)N$ monochromatic extensions. Then, by only keeping these cliques in the sum defining $E$, we would have that
	\begin{align*}
		E&=\sum_Q (\ext(Q)-2^{-k}N)^2 \geq c_2 N^k (c_2 N)^2=c_2^3 N^{k+2}.
	\end{align*}
	Therefore, if $\theta$ is small enough that $\delta< c_2^3/(2k)$, we have a contradiction.
\end{proof}

\section{Conclusion}

The outstanding open problem that remains is Thomason's conjecture~\cite{Thomason82}, that 
\[
	r(B_n \up k) \leq 2^k(n+k-2)+2.
\]
There are, in fact, two problems here, the problem of proving this conjecture for all $n$ and $k$ and the problem of proving it when $k$ is fixed and $n$ is sufficiently large. The first of these problems seems bewilderingly hard, not least because it would immediately yield an exponential improvement for the classical Ramsey number $r(K_r)$, as outlined in the introduction. The second problem may be more approachable, but we have not found a way of leveraging our stability result, saying that near extremal colorings are quasirandom, to obtain this more precise statement. 

\paragraph{Acknowledgements.} We would like to thank Freddie Illingworth for pointing out an error in an earlier draft of this paper. We would also like to thank the anonymous referees for their careful reviews and helpful suggestions.


\appendix

\section{Proofs of technical lemmas}\label{appendix}
In this appendix, we collect those proofs which did not fit neatly into the structure of the paper itself. We begin with the proof of Lemma \ref{lem:random-subgraph-reg}, which says that a random subset of a regular set is still regular, as long as the subset is not too small.
\begin{proof}[Proof of Lemma \ref{lem:random-subgraph-reg}]
	Say that a graph $G$ of density $d$ is \emph{$\varepsilon$-homogeneous} if
	\[
		\max_{S,T \subseteq V(G)}|e(S,T)-d|S||T|| \leq \varepsilon |V|^2.
	\]
	It is easy to check \cite[Exercise~9.6]{Lovasz} that $\varepsilon$-regularity implies $\varepsilon$-homogeneity, which in turn implies $\varepsilon^{1/3}$-regularity.

	Given two edge-weighted graphs $G$ and $H$ on the same set $V$ of vertices, their \emph{cut distance} is defined by
	\[
		d_\square(G,H)=\frac{1}{|V|^2}\max_{S,T \subseteq V} |e_G(S,T)-e_{H}(S,T)|,
	\]
	where $e_G$ and $e_{H}$ denote the total weight of edges between $S$ and $T$ in $G$ and $H$, respectively. Note that if $G$ has density $d$ and $H_d$ denotes the complete graph with loops on vertex set $V$ where every edge receives weight $d$, then $\varepsilon$-homogeneity is equivalent to the statement that $d_\square(G,H_d) \leq \varepsilon$. Finally, we will need the First Sampling Lemma \cite[Lemma 10.5]{Lovasz}, which says that if $U$ is chosen uniformly at random from $\binom Vt$, then, with probability at least $1-4 e^{-\sqrt t/10}$,
	\[
		\left|d_\square(G[U],H[U])-d_\square(G,H)\right| \leq \frac{8}{t^{1/4}}.
	\]
	Now suppose, as in the Lemma statement, that $G$ is $\varepsilon$-regular and let $d$ be its density. Then $G$ is $\varepsilon$-homogeneous, so $d_\square(G,H_d) \leq \varepsilon$. Let $U$ be chosen uniformly at random from $\binom Vt$ and let $d'$ denote the density of $G[U]$. Then we have that
	\begin{align}
		d_\square(G[U],H_{d'}[U])&=(d_\square(G[U],H_{d'}[U])-d_\square(G[U],H_d[U]))\notag\\
		&\hspace{1cm}+(d_\square(G[U],H_d[U])-d_\square(G,H_d))+d_\square(G,H_d)\label{eq:cutnorm-bound}
	\end{align}
	and we can bound each of these terms in turn. Since the triangle inequality holds for the cut distance, we know that
	\[
	|d_\square(G[U],H_{d'}[U])-d_\square(G[U],H_d[U])| \leq d_\square (H_d[U],H_{d'}[U]) = |d-d'|.
	\]
	To bound this, suppose we reveal the vertices in the random subset $U$ one at a time and let $Z_0,\ldots,Z_t$ be the martingale where $Z_i$ is the expected value of $e(G[U])$ conditioned on the first $i$ vertices revealed. Then $|Z_i-Z_{i-1}| \leq t-1$, since each new vertex can affect the total number of edges in $U$ by at most $t-1$, and $Z_0=\E[e(G[U])]=d\binom t2$. Therefore, by Azuma's inequality, we see that for any $\lambda>0$,
	\[
		\pr(|Z_t-Z_0| >\lambda) \leq 2e^{-\lambda^2/(2t(t-1)^2)}\leq 2 e^{-\lambda^2/(4t \binom t2)}.
	\]
	If we set $\lambda=\delta\binom t2$ for some $\delta>0$, then we see that
	\[
		\pr(|d-d'|>\delta) \leq 2 e^{-\delta^2 \binom t2/4t}=2e^{-\delta^2 (t-1)/8}.
	\]
	For the next term in (\ref{eq:cutnorm-bound}), we use the First Sampling Lemma, which implies that with probability at least $1-4e^{-\sqrt t/10}$,
	\[
		d_\square(G[U],H_d[U])-d_\square(G,H_d) \leq \frac{8}{t^{1/4}}.
	\]
	Finally, by our assumption that $G$ is $\varepsilon$-regular, we know that $d_\square(G,H_d) \leq \varepsilon$. Putting this all together, we see that for any $\delta>0$, with probability at least $1-4e^{-\sqrt t/10}-2e^{-\delta^2 (t-1)/8}$, we have that
	\[
		d_\square(G[U],H_{d'}[U]) \leq \delta+\frac{8}{t^{1/4}}+\varepsilon.
	\]
	Plugging in $\delta=t^{-1/4}$ and using our assumption that $t \geq \varepsilon^{-4}$, we find that
	\[
		d_\square(G[U],H_{d'}[U]) \leq \varepsilon+8 \varepsilon+\varepsilon=10 \varepsilon
	\]
	with probability at least $1-6 e^{-\sqrt t/10} \geq 1-6 e^{-\varepsilon^{-2}/10}$. By our assumption that $\varepsilon<1/5$, this quantity is strictly positive, so there exists some subset $U$ satisfying $d_\square(G[U],H_{d'}[U]) \leq 10 \varepsilon$. This means that $G[U]$ is $10 \varepsilon$-homogeneous, so it must also be $(10 \varepsilon)^{1/3}$-regular, as desired.
\end{proof}

The rest of the section consists of proofs of the various analytic results used throughout the paper. First, we need the following simple inequality.

\begin{lem}[Multiplicative Jensen inequality]\label{multjensen}
	Suppose $0<a<b$ 
	and $x_1,\ldots,x_k \in (a,b)$. Let $f: (a,b) \to \R$ be a function such that $y \mapsto f(e^y)$ is strictly convex on the interval $(\log a,\log b)$. Then, for any $z \in (a^k,b^k)$, subject to the constraint $\prod_{i=1}^k x_i=z$,
	\[
		\frac 1k \sum_{i=1}^k f(x_i)
	\]
	is minimized when all the $x_i$ are equal (and thus equal to $z^{1/k}$).
\end{lem}

\begin{proof}
	Define new variables $y_1,\ldots,y_k$ by $y_i=\log x_i$, so that $\sum y_i=\log z$. We now apply Jensen's inequality to the strictly convex function $f \circ {\exp}$, which says that subject to the constraint $\sum y_i=\log z$, $\sum_{i=1}^k f(e^{y_i})$ is minimized when all the $y_i$ are equal. This is equivalent to the desired result.
\end{proof}

\begin{proof}[Proof of Lemma \ref{xilemma}]
	Our proof follows the proof of \cite[Lemma 8]{Conlon}, though we give considerably more detail, particularly for small values of $k$. Set $z=\prod_{i=1}^k x_i$ and assume for the moment that every $x_i$ is in $(\frac 1k, 1)$. We will show that for every fixed $z$, the inequality is true. For this, we first claim that $\varphi: y \mapsto (1-e^y)^k$ is strictly convex on $(\log \frac 1k,0)$. To see this, note that
	\begin{align*}
		\varphi''(y)&
		=ke^y(1-e^y)^{k-2}(ke^y-1).
	\end{align*}
	For $y \in (\log \frac 1k,0)$, we have that $e^y \in (\frac 1k,1)$, so that $\varphi''(y)$ is strictly positive and $\varphi$ is strictly convex. 
	
	Therefore, by Lemma \ref{multjensen}, we get that subject to the constraint $\prod_{i=1}^k x_i =z$, the function $\frac 1k \sum_{i=1}^k (1-x_i)^k$ is minimized when all the $x_i$ are equal to $z^{1/k}$. Thus, 
	\begin{align*}
		\prod_{i=1}^k x_i + \frac 1k \sum_{i=1}^k (1-x_i)^k \geq z+(1-z^{1/k})^k=:\psi(z).
	\end{align*}
	We now claim that for all $z \in (0,1)$, $\psi(z) \geq 2^{1-k}$. To see this, note that
	\begin{align*}
		\psi'(z)&=1-(1-z^{1/k})^{k-1}z^{-\frac{k-1}{k}}.
	\end{align*}
	Setting this equal to zero and taking $(k-1)$th roots, we get that $\psi'(z)=0$ implies that
	\[
		1-z^{1/k}=z^{1/k}.
	\]
	Thus, the only critical point of $\psi$ in $(0,1)$ is at $z^{1/k}=\frac 12$ or, equivalently, $z=2^{-k}$, where we have $\psi(z)=2^{1-k}$. On the other hand, $\psi(0)=\psi(1)=1$, so this critical point must be the minimum of $\psi$ on the interval $(0,1)$, proving the desired claim.

	Thus, we have proven the lemma under the assumption that $x_i \in (\frac 1k, 1)$ for all $i$. By continuity, we can get the same result under the assumption that $x_i \in [\frac 1k, 1]$ for all $i$. So now suppose that $0 \leq x_j< \frac 1k$ for some $j$. If $k \geq 5$, then 
	\[
		\frac 1k \sum_{i=1}^k (1-x_i)^k \geq \frac 1k {(1-x_j)^k}\geq \frac 1k \left( 1- \frac 1k \right) ^k >2^{1-k},
	\]
	so we only need to check the result in the cases $k=2$, $3$, and $4$.
	
	\vspace{3mm}
	\noindent
		\textbf{Case $k=2$:} We may write
		\begin{align*}
			x_1 x_2+ \frac12 ({(1-x_1)^2+(1-x_2)^2})&=\frac 12+ \frac 12\left( 1-2x_1-2 x_2+2x_1 x_2+x_1^2+x_2^2 \right) \\
			&=\frac 12 + \frac{1}{2}(1-x_1-x_2)^2 \\
			& \geq \frac 12=2^{1-k}.
		\end{align*}
		
        \vspace{3mm}
	    \noindent
		\textbf{Case $k=3$:} The function we are trying to minimize is
		\[
			F(x_1,x_2,x_3)=x_1 x_2 x_3+ \frac 13 ((1-x_1)^3+(1-x_2)^3+(1-x_3)^3).
		\]
		To minimize this function, we will find its critical points. Its partial derivatives are
		\begin{align*}
			\pardiff F{x_i}&= \prod_{j \neq i} x_j-(1-x_i)^2.
		\end{align*}
		If we set each of these three equations equal to zero, it is tedious but straightforward to verify that the only solution in $(0,1)^3$ is at $(\frac 12, \frac 12, \frac 12)$. Thus, $F$ is either minimized at that point (where its value is $1/4=2^{1-k}$) or else on the boundary of the cube. So we may assume that one of the $x_i$, say $x_3$, is in $\{0,1\}$. If $x_3 = 0$, we have that
		\[
			F(x_1,x_2,0)=\frac 13 (1+(1-x_1)^3+(1-x_2)^3)
		\]
		and since this function is monotonically decreasing in both $x_1$ and $x_2$, it is minimized when $x_1=x_2=1$, where its value is $1/3$, which is larger than $1/4$. For the other case, where $x_3=1$, we write
		\begin{align*}
			G(x_1,x_2)=F(x_1,x_2,1)&=x_1 x_2+\frac 13 ((1-x_1)^3+(1-x_2)^3).
		\end{align*}
		The partial derivatives of $G$ are 
		\[
			\pardiff G{x_1}=x_2-(1-x_1)^2 \qquad \qquad \pardiff{G}{x_2}=x_1-(1-x_2)^2
		\]
		and setting these both equal to zero, we find that the only critical point of $G$ in $(0,1)^2$ is when $x_1=x_2=\frac 12 (3-\sqrt 5)$. But at this point the value of $G$ is approximately $0.303$, which is larger than $1/4$. So it again suffices to check the boundary case, when one of the variables, say $x_2$, is in $\{0,1\}$. But if $x_2=0$, then
		\[
			G(x_1,0)=\frac 13(1+(1-x_1)^3) \geq \frac 13
		\]
		and if $x_2=1$, then
		\[
			G(x_1,1)=x_1+\frac13 (1-x_1)^3=\frac 13+x_1^2- \frac{x_1^3}{3}\geq \frac 13+\frac{2 x_1^2}{3} \geq \frac 13,
		\]
		where we used that $x_1^3 \leq x_1^2$ for $x_1 \in [0,1]$. 

		\vspace{3mm}
	    \noindent
		\textbf{Case $k=4$:} Here, our function is
		\[
			F(x_1,x_2,x_3,x_4)=x_1 x_2 x_3 x_4+\frac 14 ((1-x_1)^4+(1-x_2)^4+(1-x_3)^4+(1-x_4)^4).
		\]
		To minimize, we again consider the partial derivatives
		\begin{align*}
			\pardiff F{x_i}= \prod_{j \neq i} x_j - (1-x_i)^3.
		\end{align*}
		Setting each equation equal to zero and multiplying by its $x_i$, we get that
		\[
			x_1(1-x_1)^3=x_2(1-x_2)^3=x_3(1-x_3)^3=x_4(1-x_4)^3=x_1 x_2 x_3 x_4.
		\]
		Setting $z=x_1 x_2 x_3 x_4$, this system of equations tells us that, for all $i$,
		\[
			x_i(1-x_i)^3=z.
		\]
		Observe that the function $x \mapsto x(1-x)^3$ is $0$ at both $0$ and $1$ and has a unique local maximum in the interval $(0,1)$. This implies that, for any fixed $z$, the equation $x(1-x)^3=z$ has at most two solutions for $x \in [0,1]$. Thus, we find that all four $x_i$ can take on at most two values at the critical point. Call these values $a$ and $b$. We now split into cases depending on how many $x_i$ take on each value. 

		If all the $x_i$ are equal, say to $a$, then we are done. Indeed, in that case $z=a^4$ and all our equations become
		\[
			a(1-a)^3=a^4,
		\]
		which implies that $a=1-a$, i.e.,\ $a=\frac 12$, and the only critical point with all its coordinates equal is $(\frac 12, \frac 12, \frac 12, \frac 12)$. 

		Next, suppose that $x_1=a$, while $x_2=x_3=x_4=b$. Then $z=ab^3$ and our equations become
		\[
			a(1-a)^3=ab^3=b(1-b)^3.
		\]
		From the first equation, we find that $(1-a)^3=b^3$ and thus that $a+b=1$. Plugging this into the second equation, we find that 
		\[
			ab^3=b(1-b)^3=ba^3
		\]
		and thus $a^2=b^2$, so that $a=b$. Combining this with $a+b=1$, we find that $a=b=\frac 12$, returning us to the previous case. 

		Finally, the remaining case is where $x_1=x_2=a$, while $x_3=x_4=b$. Then $z=a^2 b^2$ and our equations are
		\[
			a(1-a)^3=a^2b^2=b(1-b)^3.
		\]
		Multiplying the two equations together and dividing out by $ab$ gives
		\[
			(1-a)^3(1-b)^3=a^3 b^3
		\]
		and thus $(1-a)(1-b)=ab$. Expanding out the left-hand side and rearranging tells us that $a+b=1$. Therefore, our equations become
		\[
			ab^3=a^2 b^2=ba^3,
		\]
		which implies that $a=b$. This and $a+b=1$ imply that $a=b=\frac 12$, again yielding our previously found critical point. Therefore, the unique critical point of $F$ in $(0,1)^4$ is at $(\frac 12, \frac 12, \frac 12, \frac 12)$.
		
		To conclude, we again need to check what happens on the boundary. So suppose that some variable, say $x_4$, is in $\{0,1\}$. If $x_4=0$, then
		\[
		    F(x_1,x_2,x_3,0)=\frac 14 (1+(1-x_1)^4+(1-x_2)^4+(1-x_3)^4).
		\]
		This function is monotonically decreasing in $x_1,x_2,x_3$, so it is minimized when $x_1=x_2=x_3=1$, at which point we have
		\[
		    F(1,1,1,0)=\frac 14(1+0+0+0)=\frac 14>\frac 18=F \left(\frac 12, \frac12, \frac 12, \frac 12\right).
		\]
		On the other hand, if $x_4=1$, then we can define
		\[
	    G(x_1,x_2,x_3)=F(x_1,x_2,x_3,1)=x_1 x_2 x_3 +\frac 14((1-x_1)^4+(1-x_2)^4+(1-x_3)^4).
		\]
		Solving for where the gradient of $G$ is $0$, we find that the only such point in $(0,1)^3$ is approximately $(0.43,0.43,0.43)$, where the value of $G$ is approximately $0.159$, which is more than $1/8$. So the only way the value of $F$ can be smaller than $1/8$ on the boundary is if another variable, say $x_3$, is also on the boundary. If $x_3=0$, then we have
		\[
		F(x_1,x_2,0,1)=\frac 14(1+(1-x_1)^4+(1-x_2)^4),
		\]
		which is minimized when $x_1=x_2=1$, where its value is $1/4>1/8$. So we may suppose $x_3=x_4=1$ and define
		\[
		H(x_1,x_2)=F(x_1,x_2,1,1)=x_1 x_2+\frac 14((1-x_1)^4+(1-x_2)^4).
		\]
		Its gradient equals zero in $(0,1)^2$ only at approximately $(0.32,0.32)$, where its value is approximately $0.209>1/8$. So the only remaining case is the boundary, when $x_2 \in \{0,1\}$. As above, if $x_2=0$, then $H(x_1,0)\geq 1/4>1/8$, so this case is not a problem. If $x_2=x_3=x_4=1$, then
		\[
		F(x_1,1,1,1)=x_1+\frac{(1-x_1)^4}4,
		\]
		which is minimized at $x_1=0$, where its value is $1/4>1/8$. Thus, having checked the boundary of $[0,1]^4$, we can conclude that the unique minimum of $F$ on $[0,1]^4$ is at $(\frac 12, \frac 12,\frac 12,\frac 12)$.
		\end{proof}
	
In order to prove Lemma \ref{stablexilemma}, we will use H\"older's defect inequality, which can be thought of as a stability version of Jensen's inequality. The statement is as follows.
\begin{thm}[H\"older's Defect Formula, {\cite[Problem 6.5]{Steele}}]\label{holderdefect}
	Suppose $\varphi:[a,b] \to \R$ is a twice-differentiable function with $\varphi''(y) \geq m > 0$ for all $y \in (a,b)$. For any $y_1,\ldots,y_k \in [a,b]$, let
	\[
		\mu = \frac 1k \sum_{i=1}^k y_i \qquad \text{ and } \qquad \sigma^2 = \frac 1k \sum_{i=1}^k (y_i-\mu)^2
	\]
	be the empirical mean and variance of $\{y_1,\ldots,y_k\}$. Then
	\[
		\frac 1k \sum_{i=1}^k \varphi(y_i) - \varphi (\mu) \geq \frac{m\sigma^2}{2}.
	\]
\end{thm}

\begin{proof}[Proof of Lemma \ref{stablexilemma}]
	To prove this, it suffices to show that for every $k \geq 3$, the point $(\frac 12, \ldots,\frac 12)$ is the unique global minimum of the function
	\[
		F(x_1,\ldots,x_k)=\prod_{i=1}^k x_i+\frac 1k \sum_{i=1}^k (1-x_i)^k.
	\]
	For $k=3$ and $4$, we explicitly checked this in the proof of Lemma \ref{xilemma}. So from now on, we may assume that $k \geq 5$. 

	To apply Theorem \ref{holderdefect}, we need to get a uniform lower bound on $\varphi''(y)$, where $\varphi(y)=(1-e^y)^k$, as in the proof of Lemma \ref{xilemma}. To obtain this uniform lower bound we will need to restrict to a subinterval, and specifically a subinterval of $(\log \frac 1k,0)$, where $\varphi''(y) > 0$. We first deal with the case where one of the variables is not in such a subinterval. 

	Suppose first that $x_j \leq \frac {1+\xi}k$ for some $j$, where $\xi>0$ is some small constant. Then  
	\[
		\prod_{i=1}^k x_i +\frac 1k \sum_{i=1}^k (1-x_i)^k \geq \frac 1k (1-x_j)^k \geq \frac 1k \left( 1-\frac {1+\xi}k \right) ^k.
	\]
	For $k \geq 5$ and $\xi$ sufficiently small, this last term is larger than $2^{1-k}$. Therefore, as long as 
	\[
		\delta_0 \leq \frac 1k \left( 1-\frac {1+\xi}k \right) ^k -2^{1-k},
	\]
	we get the desired result whenever one of the $x_j$ is at most $\frac{1+\xi}k$. Thus, from now on, we may assume that $x_j \geq \frac{1+\xi}k$ for all $j$.
	
	Next, suppose that one of the $x_j$ is equal to $1$, say $x_k=1$. We can compute
	\[
	    \left.\pardiff F{x_k}\right|_{x_k=1}=\left[\prod_{i=1}^{k-1} x_i - (1-x_k)^{k-1}\right]_{x_k=1}=\prod_{i=1}^{k-1} x_i\geq \left( \frac{1+\xi}{k}\right)^{k-1}>0.
	\]
	Therefore, in a neighborhood of the region where $x_k=1$, we have that $F$ is a strictly increasing function of $x_k$. Now suppose that there exist some $x_1,\ldots,x_{k-1} \geq \frac{1+\xi}{k}$ so that $F(x_1,\ldots,x_{k-1},1)=2^{1-k}$. Then, by decreasing $x_k$ from $1$ to some number slightly smaller than $1$, we could decrease the value of $F$. This would then allow us to decrease the value of $F$ to below $2^{1-k}$, contradicting Lemma \ref{xilemma}. Therefore, whenever $x_k=1$, the value of $F$ must be strictly larger than $2^{1-k}$. By symmetry, the same holds if any coordinate is $1$. Since the boundary of the cube is compact, we find that $F(x_1,\ldots,x_k) \geq 2^{1-k}+2 \delta_0$, for some $\delta_0>0$, whenever one of the $x_j$ equals $1$. By continuity of $F$, this also implies a lower bound of $2^{1-k}+\delta_0$ in some neighborhood of the boundary. Thus, for $\xi'$ sufficiently small, we get the desired result if any of the $x_j$ is larger than $1-\xi'$.

	We have therefore established the desired bound if some $x_j$ is either less than $\frac{1+\xi}{k}$ or greater than $1-\xi'$, for some constants $\xi$ and $\xi'$ depending only on $k$. Inside the interval $(\frac{1+\xi}{k},1- \xi')$, we have a uniform lower bound on $\varphi''$, so, by the H\"older defect formula, Theorem \ref{holderdefect}, we also obtain the desired result in this interval. 
\end{proof}

\end{document}